\documentclass[11pt,reqno]{amsart}
\usepackage[top=1.3in,bottom=1.3in,left=1.3in,right=1.3in]{geometry}
\usepackage{color}

\usepackage{amssymb}
\usepackage{amsmath}
\usepackage{amsfonts}
\usepackage{geometry}
\usepackage{graphicx}
\usepackage{mathrsfs,amssymb}

\usepackage[utf8]{inputenc}
\usepackage[T1]{fontenc}
\usepackage{lmodern}

\usepackage{hyperref} 
\usepackage{cleveref}

\usepackage{enumerate}

\newtheorem{theorem}{Theorem}[section]
\newtheorem{proposition}[theorem]{Proposition}

\newtheorem{lemma}[theorem]{Lemma}

\theoremstyle{definition}
\newtheorem{definition}[theorem]{Definition}

\theoremstyle{remark}
\newtheorem{remark}[theorem]{Remark}

\numberwithin{equation}{section}

\renewcommand{\Re}{\operatorname{Re}}

\newcommand{\rg}{\operatorname{rg}}

\newcommand{\R}{\mathbb{R}}
\newcommand{\N}{\mathbb{N}}
\newcommand{\C}{\mathbb{C}}
\newcommand{\Z}{\mathbb{Z}}

\newcommand{\la}{\lambda}

\newcommand{\mc}[1]{\mathcal{#1}}

\title[]{Stable blowup profile for a semilinear Heat Equation with spatially inhomogeneous nonlinearity}

\author{Irfan Glogi\'c}
\address{Fakult\"at f\"ur Mathematik, Universit\"at Bielefeld, D-33501 Bielefeld, Germany}
\email{irfan.glogic@uni-bielefeld.de}

\author{Sarah Kistner}
\address{Universit\"at Innsbruck, Institut f\"ur Mathematik,\\ Technikerstraße 13, 6020 Innsbruck, Austria}
\email{Sarah.kistner@uibk.ac.at}

\author{Birgit Sch\"orkhuber}
\address{Universit\"at Innsbruck, Institut f\"ur Mathematik,\\ Technikerstraße 13, 6020 Innsbruck, Austria}
\email{Birgit.Schoerkhuber@uibk.ac.at}

\thanks{The research of I.G. was funded in whole or in part by the Austrian Science Fund (FWF) 10.55776/PAT5825523.}

\usepackage{kantlipsum}

\begin{document}
	\maketitle

\begin{abstract}
We study the focusing semilinear heat equation with an additional defocusing Hénon-type nonlinearity, the coupling of which is measured by a constant $c >0$. For $c \in (0,c^*)$, the model admits a closed-form self-similar blowup solution in every space dimension $d \geq 1$. Restricting ourselves to the three-dimensional case, we study the stability of this solution under small non-radial perturbations. By working in intersection Sobolev spaces with additional angular regularity, we prove finite co-dimension stability for all admissible values of $c$. Furthermore, we analyze the spectrum of the underlying linearized operator and we prove stable blowup for the cubic-quintic case and $c$ sufficiently close to $c^*$. Finally, we discuss the situation for small values of $c$ and use a modified version of the classical GGMT criterion to give an upper bound on the number of unstable eigenvalues. 
\end{abstract}

\section{Introduction}

\noindent We consider semilinear heat equations of the form
\begin{align} \label{1}
\partial_t u - \Delta u = |u|^{p-1}u - c |x|^{2} |u|^{2p-2}u,
\end{align}
for a real function $u = u(t,x)$, where $(t,x) \in I \times \R^d$ for some interval $I \subseteq \R$ containing zero, $p > 1$, $c \in \R$, and $d \in \N$ denoting the dimension. Eq.~\eqref{1} can be thought of as a perturbation of the focusing power nonlinearity heat equation by a so-called H\'enon-type nonlinearity, which is a spatially inhomogeneous nonlinear term that includes an increasing potential. Local well-posedness for H\'enon-type parabolic equations is well understood, see the recent works \cite{ChiIkeTan2022, ChiIdeTanTay2024} and references therein. Subsequently, the central question in this context is that of singularity formation: ~Can smooth and localized initial data lead to loss of regularity in finite time? Once finite-time blowup has been established, attention naturally turns to the classification of all possible blowup profiles. Of particular interest is the identification of generic profiles, i.e., those that persist under small perturbations. In the case of the well-studied focusing power nonlinearity heat equation
 \begin{align}\label{NLH}
 	\partial_t u - \Delta u = |u|^{p-1}u,
 \end{align}
 generic blowup is fairly well understood. Indeed, the associated ODE in $t$ (obtained by neglecting the Laplacian) admits the spatially homogeneous blowup solution
 \begin{equation}\label{ODE}
 	v_T(t) = (T-t)^{-\frac{1}{p-1}} (p-1)^{-\frac{1}{p-1}}, \quad t \in [0,T),
 \end{equation}
 which is known to be stable; see, e.g.~\cite{MerZaa97,FerMerZaa00,ColMerRap17,Har24}. This result can, in fact, be extended to more general nonlinearities that depend only on the profile $u$, as the underlying ODE structure remains present. However,  spatial inhomogeneities in the nonlinearity destroy the ODE structure, and the problem of generic blowup becomes more involved. 
 In certain cases, when the nonlinear scaling symmetry is present, one can look for self-similar solutions, which provide concrete examples of blowup, see, e.g.~\cite{FilTer2000, IagSan2023, GenWan2024} for existence results in the context of heat equations with spatially inhomogeneous nonlinearities. Whether or not a constructed self-similar blowup solution is stable depends, however, on the precise shape of the underlying similarity profile, which is typically not available in closed form. This renders the stability analysis challenging, and, to our knowledge, in all of the instances mentioned above, the question of stability remains open.

In this paper, we focus on the specific form of the spatially inhomogeneous nonlinearity in Eq.~\eqref{1}. Our choice is guided by two requirements: first, that the nonlinearity depends smoothly on $x$; and second, that the resulting equation obeys the same scaling law as its pure power counterpart \eqref{NLH}, namely
\[u \mapsto u_{\lambda}, \quad u_{\lambda}(t,x) = \lambda^{\frac{2}{p-1}} u(\lambda^2 t, \lambda x), \quad \lambda > 0.\]  

We find that for each $d \in \N$, $p>1$ and $0 < c < \frac{p}{d^2}$ there is an \textit{explicit} radial self-similar blowup solution of the following form
\begin{align} \label{profile}
u_T(t,x):= (T-t)^{-\frac{1}{p-1}} \phi \left( \frac{|x|}{\sqrt{T-t}} \right), \quad \text{where} \quad \phi(r) = \left(\frac{a}{b+r^2} \right)^{\frac{1}{p-1}},
\end{align}
with the constants
\begin{align} \label{Constants}
a: = \frac{2}{p-1} \sqrt{\frac{p}{c}}, \quad b :=2 \left( \sqrt{\frac{p}{c}} -d \right).
\end{align}

We note that $u_T$ converges pointwise to the ODE blowup \eqref{ODE} in the limit $c \to 0^+$, i.e., for fixed $t \in [0,T)$  we have

\begin{align*}
\Vert u_T(t, \cdot) - v_T(t) \Vert_{L^{\infty}_{loc}(\R^d)} \rightarrow 0 \quad \text{as} \quad c \rightarrow 0^+.
\end{align*}
Thus, the ODE blowup associated with Eq.~\eqref{NLH} represents the limiting case for the blowup described by $u_T$ for \eqref{1} as $c \to 0^{+}$. In particular, for small values of   $c >0$, Eq.~\eqref{1} can be interpreted as a regularization of Eq.~\eqref{NLH}, resulting in a spatially non-trivial blowup profile that decays at infinity.

\subsection{Main results}
The goal of this paper is to develop a functional analytic framework to rigorously analyze the nonlinear stability of $u_T$ under small, non-radial perturbations. However, we believe that the techniques developed in this paper can be applied to more general problems, e.g.~to study stability of self-similar solutions of parabolic H\'enon equations constructed by Filippas and Tertikas in \cite{FilTer2000} under suitable restrictions on the nonlinearity. 

For technical reasons, as well as for the sake of readability of the paper, we restrict ourselves to the (physically) most interesting case $d=3$, and to odd values of $p >1$. Our  approach is in the spirit of \cite{DonSch2016}, \cite{BieDonSch2016}, \cite{Glogi__2020}, \cite{GloSch2024}, \cite{MR4730409}, where the stability problem is studied in similarity coordinates via semigroup methods in intersection Sobolev spaces. However, in order to control the spatially inhomogeneous nonlinearity, we work in Sobolev spaces with  additional angular regularity, see e.g.~\cite{MR2769870}, \cite{SteRod2005}. More precisely, we define the angular intersection Sobolev space $\dot{H}^{s,1}_{\omega}(\R^3) \cap \dot{H}^{k,1}_{\omega}(\R^3)$ for $0 \leq s < \frac{3}{2} < k$ as the completion of $C^{\infty}_c(\R^3)$ with respect to the norm 
\begin{align*}
\Vert f \Vert_{\dot{H}^{s,1}_{\omega} \cap \dot{H}^{k,1}_{\omega}}^2 = \Vert f \Vert_{\dot{H}^{s} \cap \dot{H}^{k}}^2 + \sum_{1 \leq i < j \leq 3} \Vert \Omega_{ij} f \Vert_{\dot{H}^{s} \cap \dot{H}^{k}}^2,
\end{align*}

 with $\Omega_{ij}:= x_i \partial_j - x_j \partial_i$ and set
\[X_{s,k}^{\omega} := \dot{H}^{s,1}_{\omega}(\R^3) \cap \dot{H}^{k,1}_{\omega}(\R^3)\]
with suitable exponents $\frac{3}{2}- \frac{2}{p-1} < s < \frac{3}{2}$ and suitable $k \in \N$, $k > \frac{3}{2}$, see Section \ref{Sec:Funct_Setup}. Our first result proves finite co-dimension stability of $u_T$ for any admissible value of $c$. Without loss of generality, we consider perturbations of $u_T$ for $T=1$ such that $u_1(0,\cdot) = \phi(|\cdot|)$ with $\phi$ given in Eq.~\eqref{profile}.

\begin{theorem}[Finite co-dimension stability] \label{maintheorem}
Let $d=3$ and $p \in \N$, $p \geq 3$ with $p$ being odd and $0 < c < \frac{p}{d^2}$. Furthermore, assume that 
\begin{align} \label{sk}
s = \frac{3}{2} - \frac{2}{p-1} + \varepsilon \quad \text{with} \quad 0 < \varepsilon < \frac{4}{(p-1)(3p-1)} \quad \text{and} \quad k = \begin{cases} 2, \ \text{if $p=3$} \\ 4, \ \text{if $p \geq 5$} \end{cases}.
\end{align}
Then, there exists $\gamma > 0$, $N \in \N$ and functions $f_j \in X_{s,k}^{\omega}$ with $j \in \{1,...,N \}$ such that the following holds. Let $\varphi_0 \in \mc S(\R^3)$ with 
\begin{align*}
\Vert \varphi_0 \Vert_{X_{s,k}^{\omega}} \leq \gamma,
\end{align*}
then there exist constants $a_j(\varphi_0) \in \R$, $j \in \{1,...,N\}$, depending Lipschitz continuously on $\varphi_0$ such that the initial value problem \eqref{1} with data
\begin{align} \label{P1}
u(0, \cdot)&  = \phi(|\cdot|) + \varphi_0 + \sum_{j=1}^N a_j(\varphi_0) f_j
\end{align}
has a classical solution $u \in C^{\infty}([0, 1) \times \R^3)$, that blows up at the origin as $t \rightarrow 1^{-}$. Furthermore, $u$ can be decomposed in the following way
\begin{align*}
u(t,x) = (1-t)^{-\frac{1}{p-1}} \left( \phi \left( \tfrac{|x|}{\sqrt{1-t}} \right) + \tilde{\varphi} \left( t, \tfrac{x}{\sqrt{1-t}} \right) \right),
\end{align*}
where 
\begin{align} \label{Conv}
\Vert \tilde{\varphi}(t, \cdot) \Vert_{L^{\infty}(\R^3)} \lesssim \Vert \tilde{\varphi}(t, \cdot) \Vert_{X_{s,k}^{\omega}} \rightarrow 0,
\end{align}
as $t \rightarrow 1^-$.
\end{theorem}

We note that the scaling of Eq.~\eqref{1} leaves the $\dot{H}^{s_c}(\R^d)$ norm, $s_c = \frac{d}{2}-\frac{2}{p-1}$, invariant. Thus, our result applies to the energy subcritical, critical and supercritical case. 

Theorem \ref{maintheorem} proves stability of the blowup solution $u_T$ provided that the initial condition is corrected along a finite number of unstable directions. The functions $f_j$ correspond to eigenfunctions of the  underlying linearized operator. In fact, we construct a co-dimension $N$ Lipschitz manifold $\mathcal{M}$ in the initial data space containing zero such that smooth perturbations chosen from $\mathcal{M}$ leads to blowup via $u_1$, see Proposition \ref{SolutionC1}. Subsequently, we show that under the conditions of Theorem \ref{maintheorem}, we have $\varphi_0 + \sum_{j=1}^N a_j(\varphi_0) f_j \in \mc M$. 

Obviously, an improvement of this result relies on knowledge about the spectrum of the linearized operator, which in fact reduces to the question of unstable eigenvalues.  Notably, their number seems to depend (continuously) on the value of $c$, see Remark \ref{RemSpec} below. In particular, if $c$  is large enough within its admissible range, the blowup is stable. For the sake of concreteness, we only state the result for $p=3$. 

\begin{theorem}[Stable blowup] \label{maintheorem2}
Let $d=3$ and $p=3$ and choose the same set of Sobolev exponents $(s,k)$ with $\varepsilon >0$ as in \eqref{sk}. 
Then, for any $c \in (\frac{1}{4}, \frac{1}{3})$ there exists $\gamma > 0$ such that   any initial condition of the form
\begin{align*}
u(0,\cdot) = \phi(|\cdot|) + \varphi_0,
\end{align*}
with $\varphi_0 \in \mc S(\R^3)$ satisfying
\begin{align*}
\Vert \varphi_0 \Vert_{X_{s,k}^{\omega}} \leq \gamma,
\end{align*}
there exists $T > 0$ and a classical solution $u \in C^{\infty}([0, T) \times \R^3)$ to Eq.~\eqref{1}, that blows up at the origin as $t \rightarrow T^{-}$. Furthermore, $u$ can be decomposed as
\begin{align*}
u(t,x) = (T-t)^{-\frac{1}{p-1}} \left( \phi \left( \frac{|x|}{\sqrt{T-t}} \right) + \tilde{\varphi} \left( t, \frac{x}{\sqrt{T-t}} \right) \right),
\end{align*}
where
\begin{align} \label{Conv}
\Vert \tilde{\varphi} (t, \cdot) \Vert_{L^{\infty}(\R^3)} \lesssim \Vert \tilde{\varphi}(t, \cdot) \Vert_{X_{s,k}^{\omega}} \rightarrow 0,
\end{align}
as $t \rightarrow T^-$.
\end{theorem}

Eq.~ \eqref{Conv} implies that the solution $u$ from Theorem \ref{maintheorem2}, when dynamically rescaled, converges back to $\phi$ in $L^{\infty}(\R^3)$, meaning that
\begin{align*}
 \Vert (T-t)^{\frac{1}{p-1}} u(t, \sqrt{T-t} \cdot) - \phi(|\cdot|) \Vert_{L^{\infty}(\R^3)} \rightarrow 0, \quad \text{as $t \rightarrow T^-$}.
\end{align*}
The same holds true for the solution $u$ of Theorem \ref{maintheorem} with $T=1$.

\begin{remark}\label{RemSpec}  Theorem \ref{maintheorem2} is based on spectral analysis of the underlying linearised problem using decomposition into spherical harmonics. For any $c$ within the admissible range, time-translation invariance induces an unstable eigenvalue $\lambda =1$. By using supersymmetry, we investigate the spectrum modulo this symmetry eigenvalue and find that for large enough values of $c$ no additional unstable eigenvalues occur, simply due to positivity of the associated potentials. 
This explains the assumption on $c$ in Theorem \ref{maintheorem2}. In contrast, as  $c \to 0^+$, we expect a genuine unstable eigenvalue to emerge (for $\ell = 1$). This can be understood by examining the limiting case of the ODE blowup and can be justified perturbatively for small enough values of $c$. In particular, we conjecture that $u_T$ is co-dimension one unstable for $c$ close to zero. 

An upper bound on the number of unstable eigenvalues can be given by using a version of the GGMT criterion \cite{GlaMarGroThi1976} which we adapt to our problem in Appendix \ref{GGMTv}, Theorem \ref{TheoremA}. We believe that this result might be useful more generally in the context of stability analysis of self-similar solutions of nonlinear parabolic equations.  

For the problem at hand, we provide examples on upper bounds for particular choices of $c$ in Appendix \ref{Discussion} and discuss numerical evidence for the structure of the spectrum. However, due to the complicated dependence of the involved quantities on $c$ is rather challenging to prove rigorous results uniformly in $c$; we thus leave this to future work. 
\end{remark}

\begin{remark}
A key step in the proofs of Theorem \ref{maintheorem} and Theorem \ref{maintheorem2} is the translation of spectral information for the non-self-adjoint operator governing the linearized evolution into growth bounds for the corresponding semigroup. In contrast to previous works, we avoid both energy estimates (via graph norms) as in \cite{BieDonSch2016}, \cite{Glogi__2020}, \cite{GloSch2024}, \cite{MR4730409} and resolvent estimates as in \cite{DonSch2016}, and instead provide a substantially shorter and more abstract argument based on a smoothing lemma (see Lemma \ref{Lemmacompactperturbation}) and the spectral properties of the semigroup, c.f. also \cite{JiaSve2015}.

For the nonlinear part of the argument, we apply a generalized Strauss inequality for non-radial functions in weighted angular Sobolev spaces proved in \cite{MR2769870}, Corollary 1.4, which allows us to control the unbounded weight in the nonlinearity. In three space dimensions the relevant inequality (1.5) of Corollary 1.4 \cite{MR2769870} implies
\begin{align*}
\sup_{x \in \R^3} ||x|^{\frac{3}{2}-t} f(x)| \lesssim \Vert f \Vert_{\dot{H}^{t,1}_{\omega}},
\end{align*}
for all Schwartz functions $f \in \mathcal{S}(\R^3)$, where $t$ satisfies $t \in \left( \frac{1}{2}, \frac{3}{2} \right)$. We emphasize that the definition of the $\Vert \cdot \Vert_{\dot{H}^{t,1}_{\omega}}$-norm given in \cite{MR2769870} is equivalent to our norm defined in Section \ref{Sec:Funct_Setup} for every $f \in \mathcal{S}(\R^3)$. This is due to the expansion of Schwartz functions in $\dot{H}^t(\R^3)$ via spherical harmonics and defining the action of $(-\Delta_{\omega})^{\frac{1}{2}}$ via spectral resolution.
\end{remark}

\subsection{Notation and conventions}
We write $a \lesssim b$ if there exists a constant $C >0$, such that $a \leq Cb$ and we write $a \simeq b$ if $a \lesssim b$ and $b \lesssim a$. By $C^{\infty}(\R^d)$ and $\mc S(\R^d)$ we denote the space of smooth functions and the space of Schwartz functions, respectively. By $C^{\infty}_c(\R^d)$ we denote the standard test space consisting of smooth and compactly supported functions. In case of radial functions we use the lower index $r$  as in $C^{\infty}_r(\R^d)$, $\mc S_r(\R^d)$, $C^{\infty}_{c,r}(\R^d)$. For a closed linear operator $(\mathcal{L}, \mathcal{D}(\mathcal{L}))$, we write $\rho(\mathcal{L})$ for the resolvent set and $\sigma(\mathcal{L}) := \C \setminus \rho(\mathcal{L})$ for the spectrum. Given $\lambda \in \rho(\mathcal{L})$, we use the following convention for the resolvent $R_{\mathcal{L}}(\lambda) := (\lambda-\mathcal{L})^{-1}$. For $f \in C^{\infty}_c(\R^d)$ we use the following definition of the Fourier transform 
\begin{align*}
\hat{f}(\xi) = \mathcal{F}f(\xi) := (2 \pi)^{-\frac{d}{2}} \int_{\R^d} e^{-i \xi \cdot x} f(x) dx, \quad \xi \in \R^d.
\end{align*}

\section{Formulation of the problem}
We consider the initial value problem
\begin{align}\label{Eq:IVP_Pert}
\begin{split}
\partial_t u - \Delta u = |u|^{p-1}u - c |x|^{2} |u|^{2p-2}u, \\
u(0,\cdot) = \phi(|\cdot|) + v_0,
\end{split}
\end{align}
where $\phi(|\cdot|) = u_1(0,\cdot)$ denotes the profile of the self-similar solution $u_T$ for $T=1$ given in \eqref{profile}, and $v_0$ denotes a perturbation. Obviously,  $u_T$ depends on the parameter $c \in (0,\frac{p}{d^2})$. However, for the sake of readability, we omit this dependence subsequently in the notation.

\subsection{Similarity variables.} Let $T>0$, for $t \in [0,T)$ and $x \in \R^3$ we define
\begin{align} \label{Variables}
\tau = \tau(t) := \log \left( \frac{T}{T-t} \right) \quad \text{and} \quad y = y(t,x) := \frac{x}{\sqrt{T-t}}.
\end{align}
This choice of variables maps the time interval $[0,T)$ into the unbounded interval $[0, \infty)$. Next, we set $(T-t)^{\frac{1}{p-1}}u(t,x) = \psi(\tau(t), y(t,x))$ and make the ansatz 
\[\psi(\tau, y) = \phi(|y|)+\varphi(\tau,y)\]
for a function $\varphi$ to study the evolution near $\phi$.  By this, we can reformulate problem \eqref{Eq:IVP_Pert} as
\begin{align} \label{CE}
\begin{cases} \partial_{\tau} \varphi(\tau,\cdot) = L \varphi (\tau, \cdot) + \mathcal{N}(\varphi(\tau, \cdot)), \quad \tau >0, \\ \varphi(0, \cdot) = \mathcal{U}(v_0,T), \end{cases}
\end{align}
where $L= L_0 + L_1$,
\begin{align}\label{Def:L0L1}
L_0 f := \Delta f - \Lambda f, \quad L_1 f := V f,
\end{align}
with the potential $V$ given by
\begin{align}\label{Def:Potential}
V(y) = p \phi (|y|)^{p-1} - c(2p-1)|y|^2 \phi (|y|)^{2p-2},
\end{align}
and the operator $\Lambda$ defined as
\begin{align*}
\Lambda f(y) := \frac{1}{2} y \cdot \nabla f(y) + \frac{1}{p-1}f(y).
\end{align*}
The nonlinearity is given by
\begin{align} \label{N}
[\mathcal{N}(f)](y) = \eta\big(\phi(|\cdot|)+f(\cdot)\big)(y) - \eta\big(\phi(|\cdot|)\big)(y) - \eta^{\prime}\big(\phi(|\cdot|)\big)(y)f(y),
\end{align}
where we define 
\begin{align*}
[\eta(f)](y):=|f(y)|^{p-1}f(y)-c|y|^2f(y)^{2p-1},
\end{align*} 
for a function $f: \R^3 \rightarrow \R$ and $y \in \R^3$. Finally, the \textit{initial data operator} is given by
\begin{align} \label{InitialU}
\mathcal{U}(v_0,T) = T^{\frac{1}{p-1}}(\phi(\sqrt{T}|y|) +v_0(\sqrt{T}y))-\phi(|y|).
\end{align}

\subsection{Functional setup}\label{Sec:Funct_Setup}
To define a suitable function space for the evolution governed by \eqref{CE} we use homogeneous intersection Sobolev spaces equipped with additional angular regularity. For $f,g \in C^{\infty}_c(\R^3)$ we consider the homogeneous Sobolev inner product
\begin{align} \label{InnerProduct}
\langle f, g \rangle_{\dot{H}^s} := \langle | \cdot |^s \mathcal{F} f, | \cdot |^s \mathcal{F} g \rangle_{L^2(\R^3)},
\end{align}
inducing the norm $\Vert \cdot \Vert_{\dot{H}^s} := \sqrt{ \langle \cdot , \cdot \rangle_{\dot{H}^s} }$, where $s \geq 0$ and $\mathcal{F}$ is the three-dimensional Fourier transform. The intersection homogeneous Sobolev space $\dot{H}^s(\R^3) \cap \dot{H}^k(\R^3)$ for exponents $0 \leq s < \frac{3}{2} < k$ is defined as the completion of $C^{\infty}_c(\R^3)$ with respect to the norm 
\begin{align*}
\Vert f \Vert^2_{\dot{H}^s \cap \dot{H}^k} := \Vert \cdot \Vert^2_{\dot{H}^s} + \Vert \cdot \Vert^2_{\dot{H}^k}.
\end{align*}
In the following we make use of the fact that for every $m \in \N_0$ we have the equivalence
\begin{align} \label{EqNorm}
\Vert f \Vert^2_{\dot{H}^m} \simeq \sum_{\substack{\alpha \in \N_0^3 \\ |\alpha|=m}} \Vert \partial^{\alpha} f \Vert^2_{L^2(\R^3)},
\end{align}
for all functions $f \in C^{\infty}_c(\R^3)$.

Let $\Delta_{\omega} = \sum_{1 \leq i < j \leq 3} \Omega_{ij}^2$ be the Laplace-Beltrami operator on $S^2$ with $\Omega_{ij}:= x_i \partial_j - x_j \partial_i$. For functions $f \in C^{\infty}_c(\R^3)$ we define a Sobolev norm with additional angular regularity by
\begin{align*}
\Vert f \Vert_{\dot{H}^{s,1}_{\omega}}^2 := \Vert f \Vert_{\dot{H}^s}^2 + \sum_{1 \leq i < j \leq 3} \Vert \Omega_{ij} f \Vert_{\dot{H}^s}^2, \quad s \geq 0.
\end{align*}
We define the angular intersection Sobolev space $\dot{H}^{s,1}_{\omega}(\R^3) \cap \dot{H}^{k,1}_{\omega}(\R^3)$ for $0 \leq s < \frac{3}{2} < k$ as the completion of $C^{\infty}_c(\R^3)$ with respect to the norm 
\begin{align*}
\Vert f \Vert_{\dot{H}^{s,1}_{\omega} \cap \dot{H}^{k,1}_{\omega}}^2 := \Vert f \Vert_{\dot{H}^{s,1}_{\omega}}^2 + \Vert f \Vert_{\dot{H}^{k,1}_{\omega}}^2.
\end{align*}
By definition this is a Hilbert space and its elements can be identified with bounded continuous functions due to the continuous embedding into $L^{\infty}(\R^3)$ stated in the following Lemma.

\begin{lemma} \label{LemmaEmbedding}
For $s_1, s_2 >0$ with $s_1 < \frac{3}{2} < s_2$, the following continuous embeddings hold
\begin{align*}
\dot{H}^{s_1,1}_{\omega}(\R^3) \cap \dot{H}^{s_2,1}_{\omega}(\R^3) \hookrightarrow \dot{H}^{s_1}(\R^3) \cap \dot{H}^{s_2}(\R^3) \hookrightarrow L^{\infty}(\R^3).
\end{align*}
Furthermore, both spaces $\dot{H}^{s_1,1}_{\omega}(\R^3) \cap \dot{H}^{s_2,1}_{\omega}(\R^3)$ and $\dot{H}^{s_1}(\R^3) \cap \dot{H}^{s_2}(\R^3)$ are closed under multiplication. Moreover,
\begin{align} \label{Multi1}
\Vert fg \Vert_{\dot{H}^{s_1,1}_{\omega} \cap \dot{H}^{s_2,1}_{\omega}} \lesssim \Vert f \Vert_{\dot{H}^{s_1,1}_{\omega} \cap \dot{H}^{s_2,1}_{\omega}} \Vert g \Vert_{\dot{H}^{s_1,1}_{\omega} \cap \dot{H}^{s_2,1}_{\omega}},
\end{align}
for all $f,g \in \dot{H}^{s_1,1}_{\omega} \cap \dot{H}^{s_2,1}_{\omega}(\R^3)$ and 
\begin{align} \label{Multi2}
\Vert fg \Vert_{\dot{H}^{s_1} \cap \dot{H}^{s_2}} \lesssim \Vert f \Vert_{\dot{H}^{s_1} \cap \dot{H}^{s_2}} \Vert g \Vert_{\dot{H}^{s_1} \cap \dot{H}^{s_2}}, 
\end{align}
for all $f,g \in \dot{H}^{s_1} \cap \dot{H}^{s_2}(\R^3)$.
\end{lemma}

\begin{proof}
The first embedding is the trivial inclusion of $\dot{H}^{s_1,1}_{\omega} (\R^3)\cap \dot{H}^{s_2,1}_{\omega}(\R^3)$ into $\dot{H}^{s_1}(\R^3) \cap \dot{H}^{s_2}(\R^3)$. The $L^{\infty}(\R^3)$ embedding is well-known, e.g.~see \cite{Glogi__2020}. Inequality \eqref{Multi2} follows from \cite{MR3200091}, Theorem 1, and $L^{\infty}(\R^3)$-embedding. To show \eqref{Multi1}, let $f,g \in \dot{H}^{s_1,1}_{\omega}(\R^3) \cap \dot{H}^{s_2,1}_{\omega}(\R^3)$, then we have
\begin{align*}
\Vert fg \Vert^2_{\dot{H}^{s_1,1}_{\omega} \cap \dot{H}^{s_2,1}_{\omega}} &= \Vert fg \Vert^2_{\dot{H}^{s_1} \cap \dot{H}^{s_2}} + \sum_{1 \leq i < j \leq 3} \Vert (\Omega_{ij} f)g + f(\Omega_{ij} g) \Vert^2_{\dot{H}^{s_1} \cap \dot{H}^{s_2}}\\
&\lesssim \Vert f \Vert^2_{\dot{H}^{s_1} \cap \dot{H}^{s_2}} \Vert g \Vert^2_{\dot{H}^{s_1} \cap \dot{H}^{s_2}} + \sum_{1 \leq i < j \leq 3} \Vert \Omega_{ij} f \Vert^2_{\dot{H}^{s_1} \cap \dot{H}^{s_2}} \Vert g \Vert^2_{\dot{H}^{s_1} \cap \dot{H}^{s_2}}\\
&~+ \sum_{1 \leq i < j \leq 3} \Vert f \Vert^2_{\dot{H}^{s_1} \cap \dot{H}^{s_2}} \Vert \Omega_{ij} g \Vert^2_{\dot{H}^{s_1} \cap \dot{H}^{s_2}}\\
&\lesssim \Vert f \Vert^2_{\dot{H}^{s_1,1}_{\omega} \cap \dot{H}^{s_2,1}_{\omega}} \Vert g \Vert^2_{\dot{H}^{s_1,1}_{\omega} \cap \dot{H}^{s_2,1}_{\omega}},
\end{align*}
as a direct consequence of \eqref{Multi2}.
\end{proof}

Next, we define the main function space of this paper. 
\begin{definition}\label{Def:Xsk}
Let 
\begin{align*} 
s = \frac{3}{2} - \frac{2}{p-1} + \varepsilon \quad \text{with} \quad 0 < \varepsilon < \frac{4}{(p-1)(3p-1)} \quad \text{and} \quad k = \begin{cases} 2, \ \text{if $p=3$,} \\ 4, \ \text{if $p \geq 5$}. \end{cases}
\end{align*}
We define 
\[ X_{s,k}^{\omega} := \dot{H}^{s,1}_{\omega}(\R^3) \cap \dot{H}^{k,1}_{\omega}(\R^3). \]
\end{definition}

\begin{remark}
The choice of exponents $(s,k)$ in Definition \ref{Def:Xsk} allows in particular to prove local Lipschitz continuity of the nonlinearity. It turns out that for $p=3$ it suffices to choose $k=2$ as a consequence of $s < 1$.
\end{remark}

\section{Linear Operator and Semigroup theory}

In this section, we investigate the linear operator $L$. To analyse the spectrum of $L$, we will work in a weighted $L^2$-space in which the operator becomes essentially self-adjoint. This gives us the necessary information required to analyze the evolution governed by $L$ on $X_{s,k}^{\omega}$. We note that we start with complex function spaces to provide the natural framework for spectral theory. In Section \ref{Sec:NonlinearEvol}, we then restrict ourselves to real-valued functions.

\subsection{Linear operator in a weighted $L^2$-space} \label{LINEAROPERATOR}

We introduce the weighted $L^2$-space $L^2_{\sigma}(\R^3)$ with corresponding inner product
\begin{align*}
\langle f, g \rangle_{L^2_{\sigma}(\R^3)} := \int_{\R^3} f(x) \overline{g(x)} \sigma(x) dx
\end{align*}
and weight $\sigma(x) = e^{-\frac{|x|^2}{4}}$.  

We equip the operator $L$,

\[Lf= \Delta f - \Lambda f + V f\]
 with the domain $\mathcal{D}(L) := C^{\infty}_c(\R^3)$ which turns it into an unbounded, densely defined and symmetric operator on $L^2_{\sigma}(\R^3)$. In order to analyze its spectrum, we first consider the following decomposition. We define the unitary map $U: L^2(\R^3) \rightarrow L^2_{\sigma}(\R^3)$ given by
\begin{align*}
[Uf](x) = e^{\frac{|x|^2}{8}} f(x),
\end{align*}
and consider the unitary equivalent operator $A = -U^{-1} L U$ with domain $\mathcal{D}(A) = U^{-1} \mathcal{D}(L) = C^{\infty}_c(\R^3)$. For $f \in C^{\infty}_{c}(\R^3)$ this operator is given by
\begin{align*}
[Af](x) = -\Delta f(x) +\left( \frac{|x|^2}{16} - \frac{3}{4} + \frac{1}{p-1} - V(|x|) \right) f(x).
\end{align*}
This operator is essentially self-adjoint with closure $\mathcal{A} : \mathcal{D}(\mathcal{A}) \subseteq L^2(\R^3) \rightarrow L^2(\R^3)$. To see this we write $A = - \Delta +\frac{| \cdot |^2}{16} + Q$, where $Q(x) = \frac{1}{p-1} - \frac{3}{4}-V(|x|)$. By Theorem X.28 on p.184 in \cite{ReedSimon} it follows that $-\Delta + \frac{| \cdot |^2}{16}$ is essentially self-adjoint on $C^{\infty}_c(\R^3)$. Since $Q \in L^{\infty}(\R^3)$ as an operator on $L^2(\R^3)$ is $\left(- \Delta +\frac{| \cdot |^2}{16} \right)$-bounded with relative bound $0$, the Kato-Rellich Theorem  (\cite{ReedSimon}, Theorem X.12 on p.162) implies that $A$ is essentially self-adjoint on $C^{\infty}_c(\R^3)$. To get more information on the spectrum of $L$ we decompose the operator $\mathcal{A}$ into one-dimensional radial Schr\"odinger operators. For that we first decompose the Hilbert space 
\begin{align*}
L^2(\R^3) = \bigoplus_{\ell \geq 0} L^2([0,\infty),r^2 dr) \otimes \mathcal{H}_{\ell} = \bigoplus_{\ell \geq 0} \bigoplus_{m=1}^{d(\ell)} L^2([0,\infty),r^2 dr) \otimes \mathcal{H}_{\ell}^m,
\end{align*} 
where $\mathcal{H}_{\ell} = \text{span} \{ Y_{\ell 1},...,Y_{\ell d(\ell)} \}$ and $\mathcal{H}_{\ell}^m = \{Y_{\ell m} \}$ and define the corresponding projections on the orthogonal subspaces as 
\begin{align*}
P_{\ell m} : L^2(\R^3) \rightarrow L^2([0,\infty),r^2 dr) \otimes \mathcal{H}_{\ell}^m, \quad [P_{\ell m} f](r \omega) = \langle f(r \cdot), Y_{\ell m} \rangle_{S^2} Y_{\ell m}(\omega),
\end{align*} 
where $r \omega = x \in \R^3 \setminus \{ 0 \}$ with $r = |x|$ and $\omega = \frac{x}{|x|}$. This decomposition reduces the operator $\mathcal{A}$ into a family of operators. More precisely,
\begin{align*}
\mathcal{A} = \bigoplus_{\ell \geq 0} \bigoplus_{m=1}^{d(l)} \mathcal{A}_{\ell m}, \quad \mathcal{A}_{\ell m}f = \mathcal{A} f \quad \text{for $f \in \mathcal{D}(\mathcal{A}_{\ell m}) = P_{\ell m} \mathcal{D(\mathcal{A})}$}.
\end{align*}
Each operator $\mathcal{A}_{\ell m}$ in polar coordinates is given by
\begin{align*}
[\mathcal{A}_{\ell m} f](r \omega) = \left( - \Delta_r + \frac{r^2}{16} - \frac{3}{4} + \frac{1}{p-1} - V(r) + \frac{\ell ( \ell +1)}{r^2} \right) f_{\ell m}(r) Y_{\ell m}(\omega),
\end{align*}
for $f \in P_{\ell m} C^{\infty}_c(\R^3)$, where $f(r \omega) = f_{\ell m}(r) Y_{\ell m}(\omega)$ with $f_{\ell m}(r) := \langle f(r \cdot), Y_{\ell m} \rangle_{S^2}$ and $ \Delta_r $ denotes the radial Laplace operator.\\
Since each operator is of the form $\mathcal{A}_{\ell m} = \mathcal{A}_{\ell m}^r \otimes Id_{\mathcal{H}_{\ell m}}$, it suffices to investigate the radial operators $\mathcal{A}_{\ell m}^r$ with domain $\mathcal{D}(\mathcal{A}_{\ell m}^r) = P_{\ell m}^r \mathcal{D}(\mathcal{A}) \subseteq L^2([0, \infty), r^2 dr)$, where $P_{\ell m}^r f = f_{\ell m}$ and $\mathcal{A}_{\ell m}^r f = \mathcal{A}_{\ell m} f_{\ell m}$.
We define the unitary map $\mathcal V: L^2(\R^+) \rightarrow L^2_r(\R^3)$ by
\begin{align*}
[\mathcal Vu](|x|) = |S^2|^{-\frac{1}{2}} |x|^{-1} u(|x|).
\end{align*}
By this, each $\mathcal{A}_{\ell m}^r$ is unitary equivalent to a one-dimensional Schr\"odinger operator $\mathcal{B}_{\ell m} = \mathcal V^{-1} \mathcal{A}_{\ell m}^r \mathcal V$ with domain $\mathcal{D} (\mathcal{B}_{\ell m}) = \mathcal V^{-1} \mathcal{D} (\mathcal{A}_{\ell m}^r) \subseteq L^2(\R^+)$ given by
\begin{align}\label{Bl}
[\mathcal{B}_{\ell m} u](r) = -u^{\prime \prime} (r) + q_{\ell}(r) u(r), \quad q_{\ell}(r) := \frac{r^2}{16} - \frac{3}{4} + \frac{1}{p-1} -V(r) + \frac{\ell (\ell +1)}{r^2},
\end{align}
for $u \in \mathcal{V}^{-1} P_{\ell m} C^{\infty}_c(\R^3)$. Let $\ell \geq 1$, then for each $\ell, m$ the operator is  limit-point at both endpoints of the interval $(0, \infty)$. Let $B_{\ell m,c}$ be the restriction of $\mathcal{B}_{\ell m}$ to $C^{\infty}_c(\R^+)$, then this operator is essentially self-adjoint, see \cite{ReedSimon}, Theorem X.7 on p.152. Since $B_{\ell m,c} \subseteq \mathcal{B}_{\ell m}$, the closure is necessarily given by the maximal operator $\mathcal{B}_{\ell m} : \mathcal{D}(\mathcal{B}_{\ell m}) \subseteq L^2(\R^+) \rightarrow L^2(\R^+)$ with domain
\begin{align*}
\mathcal{D}(\mathcal{B}_{\ell m}) = \{ u \in L^2(\R^+) : u, u^{\prime} \in AC_{\text{loc}}(\R^+), \mathcal{B}_{\ell m} u \in L^2(\R^+) \}.
\end{align*}
Furthermore, each $\mathcal{B}_{\ell m}$ has compact resolvent.\\
\\
For $\ell = 0$, the corresponding operator $\mathcal{B}_0:= {\mathcal{B}}_{0 m}$  is regular, limit circle at zero and limit point at infinity. Since $\mathcal{D}(\mathcal{B}_0) = \mathcal{V}^{-1} \mathcal{D}(\mathcal{A_0})$, where $\mathcal{A}_0 := \mathcal{A}_{0 m}$, the domain of $\mathcal{B}_0$ is given by
\begin{align*}
\mathcal{D}(\mathcal{B}_0) = \{ u \in L^2(\R^+) : u, u^{\prime} \in AC_{\text{loc}}(\R^+), \mathcal{B}_0 u \in L^2(\R^+), \lim_{r \rightarrow 0^+} u(r) = 0 \},
\end{align*} 
see \cite{MR2499016}, Theorem 9.6 on p.187. Furthermore, $\mathcal{B}_{0}$ has compact resolvent.

In summary, by Theorem 2.23. in \cite{MR2499016}, we find that the spectrum of $\mathcal{A}$ satisfies
\begin{align*}
\sigma(\mathcal{A}) = \bigcup_{\ell, m} \sigma(\mathcal{A}_{\ell m}),
\end{align*}
and $\mathcal{A}$ has compact resolvent. By the unitary equivalence of $L$ and $A$ we can state the following result.

\begin{lemma} \label{LemmaOperatorL}
The operator $L: \mc D(L) \subset L^2_{\sigma}(\R^3) \to L^2_{\sigma}(\R^3)$ is closable with the closure $\mathcal{L} : \mathcal{D}(\mathcal{L}) \subseteq L^2_{\sigma}(\R^3) \rightarrow L^2_{\sigma}(\R^3)$ being self-adjoint and having compact resolvent. In particular,  the spectrum $\sigma \left( \mathcal{L} \right)$ consists only of real eigenvalues with finite multiplicities. 
\end{lemma}

\begin{proof}
This follows immediately from the above considerations.
\end{proof}

\begin{remark}\label{Rem:Spec}
Using the above decomposition of $\mc L$ one can analyze its spectrum in more detail. We recall that the potential $V$ implicitly depends on the parameter $c \in (0,\frac{d}{p^2})$. As outlined in Remark \ref{RemSpec}, we expect a different qualitative picture for different values of $c$. However, for every $c$ within the admissible range, we have $1 \in \sigma \left( \mathcal{L} \right)$ and this eigenvalue is induced by the time-translation symmetry of the problem. More precisely, 
\begin{align*}
\partial_T u_T(t,x) = C(T,t) (b+|x|^2)^{-\frac{p}{p-1}},
\end{align*}
for some function $C(T,t)$ depending on $T$ and $t$, and it is easy to check that $g(x) := (b +|x|^2)^{-\frac{p}{p-1}}$ satisfies $Lg=g$. Furthermore, $g \in X_{s,k}^{\omega}$ since it is radial and decays sufficiently fast at infinity. 
\end{remark}

In the following, we show that for $p=3$ and large enough values of $c$, $\la=1$ is the only unstable eigenvalue, and we refer to Appendix \ref{Discussion} for a discussion of the more general case.

\begin{proposition}\label{Prop:p3Spectrum}
Let $p=3$ and $c \in (\frac{1}{4}, \frac{1}{3})$, then the spectrum of $\mathcal{L}$ satisfies
\begin{align} \label{SpectrumofL}
\sigma(\mathcal{L}) \subset (- \infty, 0) \cup \{ 1\},
\end{align}
where the eigenspace corresponding to $\lambda =1$ is one-dimensional and spanned by $g(x) = (b+|x|^2)^{-\frac{p}{p-1}}$.
\end{proposition}

\begin{proof}
Via unitary equivalence it suffices to investigate the spectrum of the operators $\mathcal{B}_{\ell m}$ defined above. For $\ell \geq 1$, $m \in \{1,...,d(\ell)\}$, one can easily check that the potentials $q_{\ell}$ are strictly positive for $c \in (\frac{1}{4}, \frac{1}{3})$. Since each operator has compact resolvent, each corresponding spectrum consists only of eigenvalues with finite multiplicities. Furthermore, we have
\begin{align*}
\inf_{\substack{u \in \mathcal{D}(\mathcal{B}_{\ell m}) \\ \Vert u \Vert_{L^2(\R^+)} =1 }} \langle \mathcal{B}_{\ell m} u , u \rangle_{L^2(\R^+)} = \inf \sigma(\mathcal{B}_{\ell m}). 
\end{align*}
Let $u \in \mathcal{D}(\mathcal{B}_{\ell m})$ with $\Vert u \Vert_{L^2(\R^+)} =1$, then by regularity of $u$ and partial integration we find
\begin{align*}
\langle \mathcal{B}_{\ell m} u , u \rangle_{L^2(\R^+)} = \int_0^{\infty} u^{\prime}(r)^2 + q_{\ell}(r) u(r)^2 dr \geq \inf_{r \geq 0} q_{\ell}(r),
\end{align*}
which implies $\inf \sigma(\mathcal{B}_{\ell m}) \geq \inf_{r \geq 0} q_{\ell}(r) > 0$. \\
Now we consider the case $\ell = 0$. Explicit computation shows that $\lambda = -1$ is an eigenvalue of $\mathcal{B}_0$ with eigenfunction $\tilde g(r) = e^{-r^2/8} r (b+r^2)^{-\frac{p}{p-1}} \in \mathcal{D}(\mathcal{B}_0)$.  To exclude other unstable eigenvalues, we use supersymmetry  and factorize $\mathcal{B}_0 = B_{+} B_{-} -1$, where 
\begin{align*}
B_{+} = \left( -\frac{d}{dr} - \frac{\tilde{g}^{\prime}}{\tilde{g}} \right) \quad \text{and} \quad B_{-} = \left( \frac{d}{dr} - \frac{\tilde{g}^{\prime}}{\tilde{g}} \right).
\end{align*}
This leads to the operator $B_S := B_{-} B_{+}-1$, being essentially self-adjoint on $C_c^{\infty}(\R^+) \subseteq L^2(\R^+)$, with the closure $\mathcal{B}_S : \mathcal{D}(\mathcal{B}_S) \subseteq L^2(\R^+) \rightarrow L^2(\R^+)$  being isospectral (modulo $\lambda = -1$) to $\mathcal{B}_0$. For a function $u \in C^{\infty}_c(\R^+)$ the operator $\mathcal{B}_S$ is explicitly given by
\begin{align} \label{BS}
[\mathcal{B}_S u](r) = -u^{\prime \prime}(r) + q_S(r) u(r),
\end{align}
where
\begin{align}\label{qS}
q_S(r) = \frac{r^2}{16}+\frac{2}{r^2}-\frac{5}{4} +\frac{p(r^2-2)}{(p-1)(b+r^2)} + \frac{4pr^2}{(p-1)^2(b+r^2)^2}
\end{align} 
and $b = b(c)$ as defined in \eqref{Constants}. Now, one can check that $q_S > 0$ for $c \in (\frac{1}{4}, \frac{1}{3})$, which implies $\inf \sigma(\mathcal{B}_{S}) > 0$ as above.\\
To see that the eigenspace corresponding to $\lambda = -1$ is indeed one-dimensional, we consider a fundamental system $\{ w_0, w_1 \}$ of solutions to $\mathcal{B}_0 u = - u$. Since the operator is in the limit-point case at $\infty$ it suffices to investigate the behaviour of a solution at $r=0$. A power series approach implies $w_0 \sim r$ and $w_1 \sim1$ near zero, hence only the linear independent solution $w_0$ belongs to $\mathcal{D}(\mathcal{B}_0)$. This shows that every eigenfunction of $\mathcal{B}_0$ has to be a multiple of $\tilde{g}$.
\end{proof}

\subsection{Semigroup theory in $L^2_{\sigma}(\R^3)$}

First, we write 
\[ \mathcal{L} = \mathcal{L}_0 + \mc L_1,\]
where $\mc L_1: L^2_{\sigma}(\R^3) \to L^2_{\sigma}(\R^3)$ is the bounded linear operator induced by the potential term, i.e., $\mc L_1 f = V f$, see Eq.~\eqref{Def:L0L1}. In particular, $\mc D(\mathcal{L}_0)= \mc D(\mathcal{L})$. For $\tau \geq 0$, we define on $L^2_{\sigma}(\R^3)$ the family of operators 
\begin{align*}
[S_0(\tau) f](x) = e^{-\frac{1}{p-1} \tau} \left( G_{\alpha(\tau)} \ast f \right)(e^{-\frac{\tau}{2}} x), \quad \text{for $x \in \R^3$ and $f \in L^2_{\sigma}(\R^3)$},
\end{align*}
where $G_{\alpha(\tau)}(x) = (4 \pi \alpha(\tau))^{-\frac{3}{2}} e^{-\frac{|x|^2}{4 \alpha(\tau)}}$ with $\alpha(\tau) = 1-e^{-\tau}$. It is well known and easy to check  that $\left(S_0(\tau) \right)_{\tau \geq 0}$ defines a strongly continuous one-parameter semigroup on $L^2_{\sigma}(\R^3)$. Furthermore, for $f \in C^{\infty}_c(\R^3)$ one has $\frac{d}{d \tau} S_0(\tau) f = S_0(\tau) \mathcal{L}_0 f$, which shows that $\mathcal{L}_0$ is the generator of $\left(S_0(\tau) \right)_{\tau \geq 0}$. 

\begin{lemma}
The operator  $(\mathcal{L}, \mathcal{D}(\mathcal{L}))$ generates a strongly continuous semigroup $\left(S(\tau) \right)_{\tau \geq 0}$ of bounded linear operators on $L^2_{\sigma}(\R^3)$.
\end{lemma}

\begin{proof}
The claim follows from the Bounded Perturbation Theorem, see Chapter III, Theorem 1.3 in \cite{MR1721989}
\end{proof}

\subsection{The linear evolution in $X_{s,k}^{\omega}$} To connect the linear theory in $L^2_{\sigma}(\R^3)$ to the space $X_{s,k}^{\omega}$ we use the following lemma.

\begin{lemma} \label{Embedding2}
The space $X_{s,k}^{\omega}$ is continuously embedded into $L^2_{\sigma}(\R^3)$.
\end{lemma}
\begin{proof}
Since $X_{s,k}^{\omega}$ is continuously embedded into $L^{\infty}(\R^3)$ by Lemma \ref{LemmaEmbedding}, and $L^{\infty}(\R^3)$ is continuously embedded into $L^2_{\sigma}(\R^3)$ by the strong decay of the weight $\sigma$, the claim follows.
\end{proof}

We make the following observation. 

\begin{lemma} \label{VginSpace}
The potential $V$ defined in Eq.~\eqref{Def:Potential} belongs to $X_{s,k}^{\omega}$.
\end{lemma}

\begin{proof}
Since $V$ is radial, we have $\Omega_{ij} V = 0$. The potential $V$ satisfies $|\partial^{\alpha} V(x)| \lesssim \langle x \rangle^{-(2+|\alpha|)}$ for all $\alpha \in \N_0^3$ and $x \in \R^3$. This implies $\partial^{\alpha} V \in L^2(\R^3)$ for all $\alpha \in \N_0^3$ and hence $V \in X_{s,k}^{\omega}$.
\end{proof}

\begin{proposition}
The restrictions $\mathcal{L}_0^X := \mathcal{L}_0 \big|_{X_{s,k}^{\omega}}$ and $\mathcal{L}^X := \mathcal{L} \big|_{X_{s,k}^{\omega}}$ of the linear operators $\mathcal{L}_0$ and $\mathcal{L}$ to $X_{s,k}^{\omega}$ with domains $\mathcal{D}(\mathcal{L}_0^X) = \mathcal{D}(\mathcal{L}^X) = \{ f \in \mathcal{D}(\mathcal{L}) \cap X_{s,k}^{\omega} : \mathcal{L}f \in X_{s,k}^{\omega} \}$ generate strongly continuous one-parameter semigroups of bounded linear operators $(S_0^X(\tau) )_{\tau \geq 0}$ and $(S^X(\tau) )_{\tau \geq 0}$ on $X_{s,k}^{\omega}$, that are given by the restrictions of $\left( S_0(\tau) \right)_{\tau \geq 0}$ and $\left( S(\tau) \right)_{\tau \geq 0}$ to $X_{s,k}^{\omega}$, respectively, i.e., $S_0^X(\tau) = S_0(\tau) \big|_{X_{s,k}^{\omega}}$ and $S^X(\tau) = S(\tau) \big|_{X_{s,k}^{\omega}}$ for all $\tau \geq 0$.
 Furthermore,

\begin{align}\label{Est:DecayS0}
\Vert S_0(\tau) f \Vert_{X_{s,k}^{\omega}} \leq e^{- \omega_0 \tau} \Vert f \Vert_{X_{s,k}^{\omega}},
\end{align}
for all $\tau \geq 0$ and all $f \in X_{s,k}^{\omega}$, where $\omega_0 := - \frac{1}{2} \left( \frac{3}{2}-\frac{2}{p-1}-s \right) >0$.
\end{proposition}

\begin{proof}
A straightforward calculation using the explicit form of $S_0(\tau)$, in particular
\begin{align*}
\mathcal{F}(S_0(\tau) f)(\cdot) = e^{\left(\frac{3}{2}-\frac{1}{p-1} \right) \tau} e^{(1-e^{\tau})|\cdot|^2} \mathcal{F}(f)(e^{\frac{\tau}{2}} \cdot),
\end{align*}
shows that Eq.~\eqref{Est:DecayS0} holds for all $f \in C_c^{\infty}(\R^d)$. 
Furthermore, 
\begin{align*}
\Vert |\cdot|^s (\mathcal{F}(S_0(\tau) f-f)) \Vert_{L^2} = \Vert |\cdot|^s (e^{(\frac{3}{2}-\frac{1}{p-1})\tau} e^{(1-e^{\tau})|\cdot|^2} \mathcal{F}(f)(e^{\frac{\tau}{2}} \cdot) - \mathcal{F}(f)) \Vert_{L^2},
\end{align*}
and 
\begin{align*}
\Vert |\cdot|^s (\mathcal{F}(\Omega_{ij}(S_0(\tau) f-f))) \Vert_{L^2} = \Vert |\cdot|^s (e^{(\frac{3}{2}-\frac{1}{p-1})\tau} e^{(1-e^{\tau})|\cdot|^2} \mathcal{F}(\Omega_{ij} f)(e^{\frac{\tau}{2}} \cdot) - \mathcal{F}(\Omega_{ij} f)) \Vert_{L^2},
\end{align*}
for $1 \leq i < j \leq 3$. By the Dominated Convergence Theorem we infer 
\begin{align*}
\lim_{\tau \rightarrow 0^+} \Vert S_0(\tau)f-f \Vert_{X_{s,k}^{\omega}} = 0.
\end{align*}
Density of $C^{\infty}_c(\R^3)$ in $X_{s,k}^{\omega}$ implies that $S_0(\tau)$ defines a strongly continuous semigroup on $X_{s,k}^{\omega}$ and the estimate \eqref{Est:DecayS0} holds on all of $X_{s,k}^{\omega}$. Since $X_{s,k}^{\omega}$ is invariant under $\left( S_0(\tau) \right)_{\tau \geq 0}$   we find (by Proposition 2.3, Chapter II, \cite{MR1721989}) that $\mathcal{L}_0^X := \mathcal{L}_0 \big|_{X_{s,k}^{\omega}}$ with domain $\mathcal{D}(\mathcal{L}_0^X) = \{ f \in \mathcal{L}_0 \cap X_{s,k}^{\omega} : \mathcal{L}_0 f \in X_{s,k}^{\omega} \}$ generates the restricted semigroup $( S_0^X(\tau) )_{\tau \geq 0}$ in $X_{s,k}^{\omega}$, where $S_0^X(\tau) = S_0(\tau) \big|_{X_{s,k}^{\omega}}$. The  Bounded Perturbation Theorem (again see Chapter III, Theorem 1.3 in \cite{MR1721989}) implies that $\mathcal{L}^X = \mathcal{L}_0^X +\mc L_1 = (\mathcal{L}_0 + \mc L_1) \big|_{X_{s,k}^{\omega}}$  with $\mathcal{D}(\mathcal{L}^X) = \mathcal{D}(\mathcal{L}_0^X)$ generates the strongly continuous semigroup $(S^X(\tau))_{\tau \geq 0}$ on $X_{s,k}^{\omega}$, which is given by the restriction $S^X(\tau) = S(\tau) \big|_{X_{s,k}^{\omega}}$ for all $\tau \geq 0$. This follows since
\begin{align*}
\Vert \mc  L_1 f \Vert_{X_{s,k}^{\omega}} = \Vert V f \Vert_{X_{s,k}^{\omega}} \lesssim \Vert V \Vert_{X_{s,k}^{\omega}} \Vert f \Vert_{X_{s,k}^{\omega}},
\end{align*}
for all $f \in X_{s,k}^{\omega}$ by Lemma \ref{LemmaEmbedding}, where $V \in X_{s,k}^{\omega}$ by Lemma \ref{VginSpace}.
\end{proof}

\begin{lemma}
The space of Schwartz functions $\mc S(\R^3)$ is a core for $\mathcal{L}^X$.
\end{lemma}

\begin{proof}
This follows directly as $\mc S(\R^3)$ is dense in $X_{s,k}^{\omega}$ and $S^X_0(\tau)$ leaves Schwartz functions invariant, hence $\mc S(\R^3)$ is a core for $\mathcal{L}_0^X$. Since $\mathcal{L}^X$ is a bounded perturbation of $\mathcal{L}_0^X$, the space $\mc S(\R^3)$ is also a core for $\mathcal{L}^X$.
\end{proof}

We emphasize that the Bounded Perturbation Theorem yields the bound
\begin{align*}
\Vert S^X(\tau) f \Vert_{X_{s,k}^{\omega}} \leq e^{\left( - \omega_0 +\Vert \mc L_1 \Vert \right) \tau} \Vert f \Vert_{X_{s,k}^{\omega}},
\end{align*}
for all $f \in X_{s,k}^{\omega}$, with the constant $\omega_0 = - \frac{1}{2} \left( \frac{3}{2}-\frac{2}{p-1}-s \right) > 0$. But, since we can not guarantee the exponent $-\frac{1}{2} \omega_0 + \Vert \mc L_1 \Vert$ to be negative, there is special effort necessary to show exponential decay of the semigroup in $X_{s,k}^{\omega}$. In order to use the information on the spectrum of $\mc L$ in $L^2_{\sigma}(\R^3)$ we first prove the following result. 

\begin{lemma} \label{Lemmacompactperturbation}
The semigroup $(S^X(\tau))_{\tau \geq 0}$ is a compact perturbation of the free semigroup $(S^X_0(\tau))_{\tau \geq 0}$, meaning that the right-hand side of the variation of parameters formula
\begin{align*}
S^X(\tau) - S^X_0(\tau) = \int_0^{\tau} S^X(\tau - \tau^{\prime}) \mc L_1 S^X_0(\tau^{\prime}) d \tau^{\prime},
\end{align*}
is a compact operator for each $\tau \geq 0$.
\end{lemma}

\begin{proof}
Let $\tau > 0$, then by \cite{MR1721989}, Theorem C.7 it suffices to show that $S^X(\tau - \tau^{\prime}) \mc L_1 S^X_0(\tau^{\prime})$ is a compact operator for each $\tau^{\prime} \in (0, \tau)$. Furthermore, since $S^X(\tau - \tau^{\prime})$ is a linear and bounded operator, it suffices to show that $ \mc L_1 S^X_0(\tau^{\prime})$ is a compact operator for each $\tau^{\prime} \in (0, \tau)$. To prove this we use a variant of the Rellich-Kondrachov Theorem (see Theorem 10, \cite{Hanche_Olsen_2010}) and show that $[\mc L_1 S_0^X(\tau)](B_1^X)$ is relatively compact in $X_{\lfloor s \rfloor,k}^{\omega}$, where $B_1^X = \{ f \in X_{s,k}^{\omega} : \Vert f \Vert_{X_{s,k}^{\omega}} \leq 1\}$. The continuous embedding of $X_{\lfloor s \rfloor,k}^{\omega}$ into $X_{s,k}^{\omega}$ then implies relative compactness of $[\mc L_1 S_0^X(\tau)](B_1^X)$ in $X_{s,k}^{\omega}$.\\
Since the proof depends on the precise values of $s$ and $k$, we start with the case $p=3$, where we have $s < 1$, $\lfloor s \rfloor =0$ and $k=2$.
Let $K_1 := V \cdot S_0^X(\tau)(B_1^X)$, $K_2 := \Delta (V \cdot S_0^X(\tau)(B_1^X))$, $K_3^{ij} := V \cdot \Omega_{ij}(S_0^X(\tau)(B_1^X))$ and $K_4^{ij} := \Delta (V \cdot \Omega_{ij} (S_0^X(\tau)(B_1^X)))$ for $1 \leq i < j \leq 3$. To apply \cite{Hanche_Olsen_2010}, Theorem 10, we need to show that each of these subsets is bounded in $H^1(\R^3)$ and for all $\varepsilon > 0$ there exists $R >0$ such that for all $f$ belonging to the corresponding subset we have that 
\begin{align} \label{CompactnessTheorem}
\int_{|x| \geq R} |f(x)|^2 + |\nabla f(x)|^2 dx < \varepsilon.
\end{align}
Let $f \in B_1^X$ and $R>0$, then
\begin{align} \label{K1}
\int_{|x| \geq R} & |V(x)(S_0^X(\tau)f)(x)|^2 +|\nabla ( V (S_0^X(\tau)f) )(x)|^2 dx \\
&\lesssim \int_{|x| \geq R} (V(x)^2  + |\nabla V(x)|^2) (S_0^X(\tau)f)(x)^2 + V(x)^2 |\nabla (S_0^X(\tau)f)(x)|^2 dx \nonumber \\
&\lesssim (\Vert V \Vert_{L^2(B_R^c)}^2 + \Vert \nabla V \Vert^2_{L^2(B_R^c)}) \Vert S_0^X(\tau) f \Vert_{L^{\infty}}^2 + \Vert V \Vert_{L^{\infty}(B_R^c)}^2 \Vert S_0^X(\tau) f \Vert_{\dot{H}^1(\R^3)}^2 \nonumber \\
&\lesssim \frac{1}{R} \Vert S_0^X(\tau) f \Vert_{X_{s,k}^{\omega}}^2 \lesssim \frac{1}{R} e^{-\omega_0 \tau}, \nonumber
\end{align}
where we used the decay of $V$ and the $L^{\infty}$-embedding of $X_{s,k}^{\omega}$, see Lemma \ref{LemmaEmbedding}. An analogous argument shows that $K_1$ is bounded in $H^1(\R^3)$. For $K_2$ we need to estimate
\begin{align} \label{K2}
&\int_{|x| \geq R} |\Delta (V(S_0^X(\tau)f)) (x)|^2 + |\nabla(\Delta( (V(S_0^X(\tau)f))) (x)|^2 dx\\
&\lesssim \Vert V \Vert_{L^{\infty}(B_R^c)}^2 \Vert \nabla( \Delta(S_0^X(\tau)f)) \Vert_{L^2}^2 + \sum_{\substack{\alpha \in \N_0^3 \\ |\alpha| \leq 3}} \Vert \partial^{\alpha} V \Vert_{L^{2}(B_R^c)}^2 \Vert S_0^X(\tau) f \Vert_{L^{\infty}}^2 \nonumber \\
&+  \sum_{\substack{\beta \in \N_0^3 \\ |\beta| \leq 2}} \Vert \partial^{\beta} V \Vert_{L^{\infty}(B_R^c)}^2 \left( \Vert S_0^X(\tau) f \Vert_{\dot{H}^1(\R^3)}^2 + \Vert S_0^X(\tau) f \Vert_{\dot{H}^2(\R^3)}^2 \right) \lesssim \frac{1}{R} M(\tau), \nonumber
\end{align}
for a constant $M = M(\tau) >0 $ depending on $\tau$, where where we again used the decay of $V$, the $L^{\infty}$-embedding of $X_{s,k}^{\omega}$ and
\begin{align*}
\Vert \nabla( \Delta(S_0^X(\tau)f)) \Vert_{L^2}^2 &\lesssim e^{\left( \frac{3}{2} - \frac{2}{p-1} -3 \right) \tau} \sum_{m=1}^3 \Vert \partial_m G_{\alpha(\tau)} \Vert_{L^1}^2 \Vert \Delta f \Vert_{L^2}^2\\
&\lesssim m(\tau) \Vert f \Vert_{X_{s,k}^{\omega}},
\end{align*}
for some constant $m(\tau) >0$. Furthermore, boundedness of $K_2$ in $H^1(\R^3)$ follows by analogous estimates. Let $1 \leq i < j \leq 3$, then the argumentation for $K_3^{ij}$ is the same as for $K_1$ by replacing $S_0^X(\tau) f$ in \eqref{K1} by $\Omega_{ij} S_0^X(\tau) f$. For $K_4^{ij}$ we do the same as for $K_2$ by replacing $S_0^X(\tau) f$ in \eqref{K2} by $\Omega_{ij} S_0^X(\tau) f$. More precisely, by using that $\Omega_{ij}$ commutes with the Laplace operator and the Fourier transform, we get  
\begin{align*}
\Vert \nabla( \Delta(\Omega_{ij} S_0^X(\tau)f)) \Vert_{L^2}^2 &\lesssim e^{\left( \frac{3}{2} - \frac{2}{p-1} -3 \right) \tau} \Vert |\cdot| \mathcal{F}(G_{\alpha(\tau)}) \Vert_{L^{\infty}}^2 \Vert \Delta \Omega_{ij} f \Vert_{L^2}^2\\
&\lesssim \overline{m}(\tau) \Vert f \Vert_{X_{s,k}^{\omega}},
\end{align*}
for some constant $\overline{m}(\tau) > 0$. Now, \cite{Hanche_Olsen_2010}, Theorem 10,  implies that each subset $K_1$, $K_2$, $K_3^{ij}$ and $K_4^{ij}$ for $1 \leq i < j \leq 3$ is relatively compact in $L^2(\R^3)$, meaning that $[\mc L_1 S_0^X(\tau)](B_1^X)$ is relatively compact in $X_{\lfloor s \rfloor, k}^{\omega}$ by definition of the norm.\\
For $p \geq 5$, where the Sobolev exponents are given by $1 < s < \frac{3}{2}$ and $k=4$, the proof works similarly by defining suitable subsets $K_n$ and $K_n^{ij}$ for $n \in \{ 1,2 \}$ and $1 \leq i < j \leq 3$ with higher order derivatives that are bounded in $H^1(\R^3)$ and satisfy \eqref{CompactnessTheorem}.
\end{proof}

Based on the previous result for the semigroup, we are now able to characterize the spectrum of the generator in $X_{s,k}^{\omega}$ in more detail.

\begin{proposition} \label{K0}
The set $K_0 := \sigma(\mathcal{L}^X) \cap \{ \lambda \in \C : \Re(\lambda) \geq 0 \}$ consists of finitely many eigenvalues of $\mathcal{L}^X$, all of which are real and have finite algebraic multiplicity.
\end{proposition}
\begin{proof}
Since $e^{\tau \sigma(\mathcal{L}_0^X)} \subseteq \sigma(S^X_0(\tau))$, see \cite{MR1721989}, Theorem 3.6 on p.276, and $r(S^X_0(\tau)) = \sup \{ |\lambda| : \lambda \in \sigma(S^X_0(\tau)) \} \leq e^{-\omega_0 \tau} \leq 1$, see \cite{MR1721989}, Proposition 2.2 on p.251, we find that $\sigma(\mathcal{L}_0^X) \subseteq \{ \lambda \in \C : \Re (\lambda) \leq - \omega_0 \}$. The fact that $S^X(\tau)$ is a compact perturbation of $S_0^X(\tau)$ by Lemma \ref{Lemmacompactperturbation} implies that $\sigma(S^X(\tau)) \setminus \sigma(S^X_0(\tau))$ consists only of isolated eigenvalues. To see this let $\lambda \in \rho(S^X_0(\tau))$, then we have
\begin{align*}
\lambda - S^X(\tau) = (\lambda - S^X_0(\tau))(Id - (\lambda - S^X_0(\tau))^{-1} (S^X(\tau)-S^X_0(\tau)).
\end{align*}
By the compactness of $(S^X(\tau)-S^X_0(\tau))$ and the analytic Fredholm theorem, see \cite{Simon15_Part4}, Theorem 3.4.13 on p.~194, it follows that $\lambda$ belongs to $\sigma(S^X(\tau)) \setminus \sigma(S^X_0(\tau))$ if and only if $\lambda$ is an isolated eigenvalue of $S^X(\tau)$ and its algebraic multiplicity is finite.\\
Next, for $\tau > 0$ let $\mu \in \sigma(S^X(\tau)) \setminus D_1$, where $D_1 := \{ z \in \C : |z| < 1 \}$. Then $\mu \in \sigma(S^X(\tau)) \setminus \sigma(S^X_0(\tau))$, meaning $\mu \in \sigma_p(S^X(\tau)) \setminus \{ 0 \} = e^{\tau \sigma_p(\mathcal{L}^X)}$, see \cite{MR1721989}, Theorem 3.7 on p.277. Hence, there exists $\lambda \in \sigma_p(\mathcal{L}^X)$ with $\mu = e^{\tau \lambda}$ and $1 \leq | \mu | = e^{\Re(\lambda) \tau}$ implying $\Re(\lambda) \geq 0$. This shows $\sigma(S^X(\tau)) \setminus D_1 \subseteq \{ e^{\tau \lambda} : \lambda \in K_0 \}$. Conversely, $\{ e^{\tau \lambda} : \lambda \in K_0 \} \subseteq \sigma(S^X(\tau)) \setminus D_1$ is true by definition. In summary, we have 
\begin{align*}
\{ e^{\tau \lambda} : \lambda \in K_0 \} = \sigma(S^X(\tau)) \setminus D_1 \subseteq \sigma(S^X(\tau)) \setminus \sigma(S^X_0(\tau)) \subseteq \sigma_p(S^X(\tau)) \setminus \{ 0 \} = e^{\tau \sigma_p(\mathcal{L}^X)},
\end{align*}
hence $K_0 \subseteq \sigma_p(\mathcal{L}^X)$. Furthermore, the set $\{ e^{\tau \lambda} : \lambda \in K_0 \}$ is finite, since its possible accumulation points necessarily lie in $\sigma(S_0^X(\tau)) \subseteq D_1$. This, in turn, implies that $K_0$ is finite. Indeed, otherwise there would be an isolated eigenvalue $\mu \in \{ e^{\lambda} : \lambda \in K_0 \}$, to which infinitely many $\lambda \in K_0$ are mapped via $\lambda \mapsto e^{\lambda} = \mu$. This means, there exist $\alpha \in K_0$ and an infinite set $I \subseteq \Z$ with $\{ \alpha + 2 \pi i k : k \in I \} \subseteq K_0$. If we choose $\tau = \sqrt{2}$, then it follows that $\{ e^{\sqrt{2}(\alpha + 2 \pi i k)} : k \in I \} \subseteq \{ e^{\sqrt{2} \lambda} : \lambda \in K_0 \}$ has to be finite, which is a contradiction.\\
That all eigenvalues of $\mathcal{L}^X$ in $K_0$ are real, follows directly from Lemma \ref{LemmaOperatorL}.
\end{proof}

In the following, we denote by $\{ \lambda_j \}_{1 \leq j \leq \tilde{N}}$ the finite set $K_0$ from Proposition \ref{K0}. For each eigenvalue $\lambda_j$, we can choose a finite basis of real-valued eigenfunctions $\{ f_j^{i} \}_{1 \leq i \leq N_j}$ for some $N_j \in \N$. We define Riesz-projections corresponding to each eigenvalue $\lambda_j$ for $j \in \{1,...,\tilde{N} \}$ by
\begin{align*}
P_j := \frac{1}{2 \pi i} \int_{\gamma_j} R_{\mathcal{L}^X}(\lambda^{\prime}) d \lambda^{\prime},
\end{align*}
where $\gamma_j$ is a positively oriented circle centered at $\lambda_j$, contained in $\rho(\mathcal{L}^X)$ and which, apart from $\lambda_j$, contains no other spectral points in its interior. In particular, $\gamma_j(s) = \lambda_j + r_j e^{2 \pi i s}$, with $s \in [0,1]$ and $r_j>0$ sufficiently small.\\
We observe that each Riesz-projection is the restriction of the corresponding Riesz-projection in $L^2_{\sigma}(\R^3)$, which is due to the fact that the resolvent of $\mathcal{L}^X$ is the restriction of the resolvent of $\mathcal{L}$ to $X_{s,k}^{\omega}$. Since each eigenfunction $f_j^{i}$ is also an eigenfunction of the self-adjoint operator $\mathcal{L}$ in $L^2_{\sigma}(\R^3)$ to the eigenvalue $\lambda_j$ the projections are given by
\begin{align*}
P_j f = \sum_{i=1}^{N_j} P_j^{i} f, \quad \text{with} \quad P_j^{i} f = \alpha_j^{i} \langle f, f_j^{i} \rangle_{L^2_{\sigma}(\R^3)} f_j^{i}, \quad \alpha_j^{i} \in \R \setminus \{ 0 \},
\end{align*}
for $f \in X_{s,k}^{\omega}$.
By definition, it follows that 
\begin{align} \label{ProjectionPj}
\text{rg}(P_j) = \text{span}(f_j^1,...,f_j^{N_j}),
\end{align}
for all $j \in \{ 1,...,\tilde{N} \}$. Furthermore, the projections are mutually transversal with $P_j P_l =0$ for $j \neq l$.
In the following we define the projection onto all unstable directions by
\begin{align}
P := \sum_{j=1}^{\tilde{N}} \sum_{i=1}^{N_j} P_j^{i}.
\end{align}

\begin{lemma} \label{K0}
Let $p=3$ and $c \in (\frac{1}{4}, \frac{1}{3})$. Then we have 
\[ \sigma_p(\mathcal{L}^X) \cap \{ \lambda \in \C : \Re(\lambda) \geq 0 \} = \{ 1 \},\] where the geometric and algebraic eigenspaces corresponding to the eigenvalue $\lambda =1$ coincide, are one-dimensional and spanned by $g(x) = (b+|x|^2)^{-\frac{p}{p-1}}$.
\end{lemma}

\begin{proof}
First, we know that $g \in X_{s,k}^{\omega}$. Furthermore, $g \in \mathcal{D}(\mathcal{L})$ by Proposition \ref{Prop:p3Spectrum} and it is straightforward to show $\mathcal{L} g \in X_{s,k}^{\omega}$, which implies $g \in \mathcal{D}(\mathcal{L}^X)$ and we have $\mathcal{L}^X g = g$. This shows $\{ 1 \} \subseteq \sigma_p(\mathcal{L}^X) \cap \{ \lambda \in \C : \Re(\lambda) \geq 0 \}$.
Next, let $\lambda \in \sigma_p(\mathcal{L}^X) \cap \{ \lambda \in \C : \Re(\lambda) \geq 0 \}$, then $\lambda \in \sigma_p(\mathcal{L})$, since $\mathcal{L}^X$ is the restriction of $\mathcal{L}$ to $X_{s,k}^{\omega}$. From \eqref{SpectrumofL} we know $\sigma(\mathcal{L}) \subset (-\infty, 0) \cup \{ 1 \}$, implying that $\lambda = 1$. That the geometric and algebraic eigenspace to the eigenvalue $\lambda = 1$ coincide and are one-dimensional follows from \eqref{ProjectionPj} and Proposition \ref{Prop:p3Spectrum}.
\end{proof}
 
Now, we are able to formulate the main result on the linearized evolution.

\begin{proposition} \label{Semigroupdec}
There exist $\nu > 0$ and $C \geq 1$, such that
\begin{align} \label{Semigroupdecay}
\Vert S^X(\tau)(1-P)f \Vert_{X_{s,k}^{\omega}} \leq C e^{-\nu \tau} \Vert (1-P)f \Vert_{X_{s,k}^{\omega}},
\end{align}
for all $\tau \geq 0$ and $f \in X_{s,k}^{\omega}$.
\end{proposition}

\begin{proof}
We note that $\ker(P)$ and $\rg(P)$ are closed subspaces of $X_{s,k}^{\omega}$ with $X_{s,k}^{\omega} = \ker(P) \oplus \rg(P)$. Furthermore, both subspaces $\ker(P)$ and $\rg(P)$ are invariant under the action of $\mathcal{L}^X$. Therefore, we can define the restricted operators $\mathcal{L}^X_1 : \mathcal{D}(\mathcal{L}^X_1) \subseteq \ker(P) \rightarrow \ker(P)$ with domain $\mathcal{D}(\mathcal{L}^X_1) = \mathcal{D}(\mathcal{L}^X) \cap \ker(P)$ given by $\mathcal{L}_1^X f = \mathcal{L}^X f$ for $f \in \mathcal{D}(\mathcal{L}^X_1)$ and $\mathcal{L}^X_2 : \mathcal{D}(\mathcal{L}_2^X) \subseteq \rg(P) \rightarrow \rg(P)$ with domain $\mathcal{D}(\mathcal{L}^X_2) = \mathcal{D}(\mathcal{L}^X) \cap \rg(P)$ given by $\mathcal{L}^X_2 f = \mathcal{L}^Xf$ for $f \in \mathcal{D}(\mathcal{L}^X_2)$. The following spectral decomposition holds 
\begin{align*}
\sigma(\mathcal{L}_1^X) = \sigma(\mathcal{L}^X) \setminus K_0 \quad \text{ and} \quad \sigma(\mathcal{L}_2^X) = K_0,
\end{align*}
see \cite{Kat95}, Theorem 6.17 on p.178. The projection $P$ commutes with $\mathcal{L}^X$ and its resolvent and so it also commutes with the semigroup $S^X(\tau)$ for all $\tau \geq 0$, see \cite{Kat95}, Theorem 6.5 on p.173. The invariance of $\ker(P)$ and $\rg(P)$ under the operators $\mathcal{L}_1^X$ and $\mathcal{L}_2^X$ respectively imply that $(S_1^X(\tau))_{\tau  \geq 0}$ with $S^X_1(\tau) := S^X(\tau) \big|_{\ker(P)}$ and $(S^X_2(\tau))_{\tau \geq 0}$ with $S_2^X(\tau) := S^X(\tau) \big|_{\rg(P)}$ are strongly continuous semigroups on $X_{s,k}^{\omega}$ with generators $\mathcal{L}_1^X$ and $\mathcal{L}_2^X$ respectively, see \cite{MR1721989}, Section 2.3 on p.60. Furthermore, we have $\rho(S^X(\tau)) \subseteq \rho(S^X_1(\tau))$ and hence $\sigma(S^X_1(\tau)) \subseteq \sigma(S^X(\tau))$ for $\tau \geq 0$, since the resolvent of $S_1^X(\tau)$ is given by the restriction of the resolvent of $S^X(\tau)$ on $\ker(P)$.\\
Now, we need to show that the growth bound of $\mathcal{L}_1^X$ is negative, i.e. $\omega_0(\mathcal{L}_1^X) < 0$. We argue by contradiction. If $\omega_0(\mathcal{L}_1^X) = \kappa  \geq 0$, then $r(S^X_1(\tau)) = \sup \{|\mu| : \mu \in \sigma(S^X_1(\tau)) \} = e^{\kappa \tau} \geq 1$. This means, there exists $\mu \in \sigma(S^X_1(\tau)) \subseteq \sigma(S^X(\tau))$ with $|\mu| \geq 1$, implying $\mu \in \sigma(S^X(\tau)) \setminus D_1 = \{ e^{\lambda \tau} : \lambda \in K_0\}$. Hence, there exists an eigenvalue $\lambda \in \C$ of $\mathcal{L}^X$ with $\Re(\lambda) \geq 0$ satisfying $\mu = e^{\lambda \tau}$ and $\mu$ is an eigenvalue of $S_1^X(\tau)$.\\
Since $\sigma_p(S^X_1(\tau)) \setminus \{ 0 \} = e^{\sigma_p(\mathcal{L}_1^X) \tau}$, see \cite{MR1721989}, Theorem 3.7 on p.277, there also exists $\lambda_1 \in \sigma_p(\mathcal{L}_1^X)$ with $\mu = e^{\lambda_1 \tau}$, which implies $\Re(\lambda_1) \geq 0$. But this contradicts the fact $\sigma(\mathcal{L}_1^X) = \sigma(\mathcal{L}^X) \setminus K_0$.\\
By definition of the growth bound we infer the existence of $\nu > 0$ and $C \geq 1$ such that \eqref{Semigroupdecay} holds.
\end{proof}

\section{The Nonlinear Time Evolution}\label{Sec:NonlinearEvol}
In the rest of the paper, we will only consider real-valued functions in $X_{s,k}^{\omega}$ and we use the same notation for the corresponding real subspace.
\subsection{Estimates for the nonlinearity and the initial data operator}
To get a better understanding of the nonlinearity we recall that $p$ is integer and odd. This allows us to rewrite the nonlinearity as
\begin{align} \label{Rewrite}
[\mathcal{N}(f)](y) = \sum_{n=2}^p \binom{p}{n} \phi(|y|)^{p-n} f(y)^n -c |y|^2 \sum_{n=2}^{2p-1} \binom{2p-1}{n} \phi(|y|)^{2p-1-n}f(y)^n.
\end{align} 
To handle the nonlinearity in $X_{s,k}^{\omega}$, it will be necessary to estimate both products of functions and weighted products of functions with weight $|\cdot|^2$. Since the first is easily done by using the Banach algebra property of $X_{s,k}^{\omega}$ given by Lemma \ref{LemmaEmbedding}, \eqref{Multi1}, the second is far more challenging. The main task is to find suitable estimates in $X_{s,k}^{\omega}$ for terms of the form 
\begin{align} \label{Terms}
|\cdot|^2 \prod_{l=1}^{2p-1} f_l, \quad f_l \in X_{s,k}^{\omega}.
\end{align}
As it turns out, the lower Sobolev exponent $s$ is not an integer, making direct estimates of the $\Vert \cdot \Vert_{\dot{H}^{s,1}_{\omega}}$-norm difficult. To overcome this issue we work in the larger function space $\dot{H}^{\lfloor s \rfloor,1}_{\omega}(\R^3) \cap \dot{H}^{k,1}_{\omega}(\R^3)$ that embeds continuously into $X_{s,k}^{\omega}$ by interpolation and allows for exploiting the equivalence or norms
\begin{align}
\Vert f \Vert^2_{\dot{H}^{m,1}_{\omega}} \simeq \sum_{\substack{\alpha \in \N_0^3 \\ |\alpha|=m}} \left( \Vert \partial^{\alpha} f \Vert^2_{L^2(\R^3)} + \sum_{i<j}  \Vert \partial^{\alpha} (\Omega_{ij} f) \Vert^2_{L^2(\R^3)} \right), \quad m \in \{ \lfloor s \rfloor, k \},
\end{align}
for $ f \in C_c^{\infty}(\R^3)$. This reduces matters to the problem of splitting all partial derivatives of \eqref{Terms} in $L^2(\R^3)$ into parts that can be estimated by a number of inequalities. The key role is played by Corollary 1.4 in \cite{MR2769870}, a generalized Strauss inequality for non-radial functions in angular Sobolev spaces. We apply it in three space-dimensions in the form
\begin{align*}
\sup_{x \in \R^3} ||x|^{\frac{3}{2}-t} f(x)| \lesssim \Vert f \Vert_{\dot{H}^{t,1}_{\omega}},
\end{align*}
for Schwartz functions $f \in \mathcal{S}(\R^3)$, where $t$ satisfies $t \in \left( \frac{1}{2}, \frac{3}{2} \right)$. We also use Hardy's inequality
\begin{align*}
\Vert |\cdot|^{-t} f \Vert_{L^2(\R^3)} \lesssim \Vert f \Vert_{\dot{H}^t}, \quad 0 \leq t < \frac{3}{2},
\end{align*}
for $f \in \mathcal{S}(\R^3)$, see \cite{MusSch13}, Theorem 9.5, p. 243. Furthermore, we make extensive use of the $L^{\infty}(\R^3)$-embedding for (angular) intersection Sobolev spaces given by Lemma \ref{LemmaEmbedding} and interpolation of Sobolev norms 
\begin{align*}
\Vert f \Vert_{\dot{H}^{s_3,1}_{\omega}} \lesssim \Vert f \Vert_{\dot{H}^{s_1,1}_{\omega} \cap \dot{H}^{s_2,1}_{\omega}},
\end{align*}
for $f \in C^{\infty}_c(\R^3)$, where $0 \leq s_1 \leq s_3 \leq s_2$. Another important property of Sobolev norms is given in the following Lemma.

\begin{lemma}
Let $\alpha \in \N_0^3$. Then we have 
\begin{align} \label{DerivativesinH}
\Vert \partial^{\alpha} f \Vert_{\dot{H}^{s,1}_{\omega}} \lesssim \Vert f \Vert_{\dot{H}^{s+|\alpha|,1}_{\omega}} ,
\end{align}
for all $f \in C^{\infty}_c(\R^3)$.
\end{lemma}

\begin{proof}
Let $f \in C^{\infty}_c(\R^3)$ and $\alpha \in \N_0^3$, then we have
\begin{align*}
\Vert \partial^{\alpha} f \Vert_{\dot{H}^{s,1}_{\omega}}^2 &= \Vert \partial^{\alpha} f \Vert_{\dot{H}^{s}}^2 + \sum_{i < j} \Vert \Omega_{ij} \partial^{\alpha} f \Vert_{\dot{H}^{s}}^2,
\end{align*}
where 
\begin{align*}
\Vert \partial^{\alpha} f \Vert_{\dot{H}^{s}} =\Vert |\cdot|^s \mathcal{F}(\partial^{\alpha} f) \Vert_{L^2} \lesssim \Vert |\cdot|^{s+|\alpha|} \mathcal{F}f \Vert_{L^2} = \Vert  f \Vert_{\dot{H}^{s+|\alpha|}}.
\end{align*}
For the second term we first observe
\begin{align*}
\Omega_{ij} \partial^{\alpha} f = \partial^{\alpha} \Omega_{ij} f - \alpha_i \partial^{\alpha -e_i + e_j} f + \alpha_j \partial^{\alpha -e_j +e_i} f.
\end{align*}
 This implies
\begin{align*}
\Vert \Omega_{ij} \partial^{\alpha} f \Vert_{\dot{H}^{s}} \lesssim \Vert \partial^{\alpha} \Omega_{ij} f \Vert_{\dot{H}^{s}} + \Vert f \Vert_{\dot{H}^{s+|\alpha|}} \lesssim \Vert f \Vert_{\dot{H}^{s+|\alpha|,1}_{\omega}},
\end{align*}
which shows \eqref{DerivativesinH}.

\end{proof}

The following lemma provides the key nonlinear estimate; it shows how the additional angular regularity assumed in the functional framework tames the growing weight appearing in the nonlinear operator \eqref{N}. For the statement, we recall the assumption on $(p,s,k)$ formulated in Definition \ref{Def:Xsk}.

\begin{lemma} \label{WProductLemma}
One has
\begin{align} \label{Ineq1}
\Vert |\cdot|^2 \prod_{l=1}^{2p-1} f_l \Vert_{X_{s,k}^{\omega}} \lesssim \prod_{l=1}^{2p-1} \Vert f_l \Vert_{X_{s,k}^{\omega}}
\end{align}
for all $f_l \in X_{s,k}^{\omega}$, $l=1,...,2p-1$.
\end{lemma}

\begin{proof}
Recall that the lower index of our Sobolev space $X_{s,k}^{\omega}$ is given by $s=\frac{3}{2} - \frac{2}{p-1} + \varepsilon$ for some small $\varepsilon >0$. For $p=3$ we find $s = \frac{1}{2} + \varepsilon < 1$ for small enough $\varepsilon >0$. For $p \geq 5$ we have $s = \frac{3}{2} - \frac{2}{p-1} + \varepsilon \geq 1 + \varepsilon$. Instead of proving \eqref{Ineq1} directly we show
\begin{align*}
\Vert |\cdot|^2 \prod_{l=1}^{2p-1} f_l \Vert_{\dot{H}^{ \lfloor s \rfloor ,1}_{\omega} \cap \dot{H}^{k,1}_{\omega}} \lesssim \prod_{l=1}^{2p-1} \Vert f_l \Vert_{X_{s,k}^{\omega}},
\end{align*}
for all $f_l \in X_{s,k}^{\omega}$, $l=1,...,2p-1$. The embedding $\dot{H}^{ \lfloor s \rfloor ,1}_{\omega} \cap \dot{H}^{k,1}_{\omega}(\R^3) \hookrightarrow X_{s,k}^{\omega}$ then implies \eqref{Ineq1}. Furthermore, by density and the $L^{\infty}$-embedding it suffices to show the inequality only for functions belonging to $C^{\infty}_c(\R^3)$.\\
We treat the case $p=3$ as a special case and start the proof with $p \geq 5$. In this case we have $\lfloor s \rfloor = 1$ and $k = 4$. Let $f_l \in C^{\infty}_c(\R^3)$, $l=1,...,2p-1$. By definition of the underlying Sobolev norm we have to estimate terms of the following form in $L^2(\R^3)$
\begin{align*}
\partial^{\alpha} \left( |\cdot|^2 \prod_{l=1}^{2p-1} f_l \right) \quad \text{and} \quad \partial^{\alpha} \left( |\cdot|^2 (\Omega_{ij} f_1) \prod_{l=2}^{2p-1} f_l \right), \quad \text{for} \ \alpha \in \N_0^3 \ \text{with} \ |\alpha| \in \{ 1, 4 \}.
\end{align*}
We only show how to estimate the terms containing angular derivatives $\Omega_{ij}$, as those without (angular derivatives) follow by the same arguments. For $|\alpha| =1$ it suffices to find estimates in $L^2(\R^3)$ for
\begin{align*}
|\cdot| (\Omega_{ij} f_1) \prod_{l=2}^{2p-1} f_l \quad , \quad |\cdot|^2 (\partial_m \Omega_{ij} f_1) \prod_{l=2}^{2p-1} f_l \quad \text{and} \quad |\cdot|^2 (\Omega_{ij} f_1) (\partial_m f_2) \prod_{l=3}^{2p-1} f_l,
\end{align*}
for some $m \in \{ 1,2,3 \}$. We apply Hardy's inequality and infer for the first term
\begin{align*}
\Vert | \cdot | (\Omega_{ij} f_1) \prod_{l=2}^{2p-1} f_l \Vert_{L^2} &\leq \Vert |\cdot|^{-s} \Omega_{ij} f_1 \Vert_{L^2} \Vert |\cdot|^{s+1}  \prod_{l=2}^{2p-1} f_l \Vert_{L^{\infty}}\\
&\lesssim \Vert f_1 \Vert_{\dot{H}^{s,1}_{\omega}} \left( \prod_{l=2}^{2p-1} \Vert f_l \Vert_{L^{\infty}(B_1)} + \Vert |\cdot|^{s+1}  \prod_{l=2}^{2p-1} f_l \Vert_{L^{\infty}(B_1^{c})} \right).
\end{align*}
By the $L^{\infty}$-embedding we find
\begin{align*}
\prod_{l=2}^{2p-1} \Vert f_l \Vert_{L^{\infty}(B_1)} \lesssim \prod_{l=2}^{2p-1} \Vert f_l \Vert_{X_{s,4}^{\omega}}.
\end{align*}
To estimate the weighted product in $L^{\infty}(B_1^c)$ we use that $|x|^{s+1} \leq |x|^3$ for all $x \in \R^3 \setminus B_1$ and do the following
\begin{align} \label{3}
\Vert |\cdot|^{3}  \prod_{l=2}^{2p-1} f_l \Vert_{L^{\infty}(B_1^{c})} &= \Vert \left( |\cdot|^{\frac{2}{p-1}-\varepsilon} \right)^{3 \frac{(p-1)}{2}} |\cdot|^{3 \frac{(p-1)}{2} \varepsilon} \prod_{l=2}^{2p-1} f_l \Vert_{L^{\infty}(B_1^{c})} \nonumber \\ 
&\leq \prod_{l=2}^{\frac{3p+1}{2}} \Vert |\cdot|^{\frac{2}{p-1}-\varepsilon} f_l  \Vert_{L^{\infty}(B_1^{c})} \prod_{n=\frac{3p+3}{2}}^{2p-1} \Vert f_n \Vert_{L^{\infty}(B_1^{c})} \lesssim \prod_{l=2}^{2p-1} \Vert f_l \Vert_{X_{s,4}^{\omega}},
\end{align}
where we applied Corollary 1.4 from \cite{MR2769870}, and exploited that $3 \frac{(p-1)}{2} \varepsilon \leq \frac{2}{p-1}-\varepsilon$. For the second term we use a similar argument
\begin{align*}
\Vert |\cdot|^2 (\partial_m \Omega_{ij} f_1) \prod_{l=2}^{2p-1} f_l \Vert_{L^2} &\leq \Vert |\cdot|^{-(s-1)} \partial_m \Omega_{ij} f_1 \Vert_{L^2} \Vert |\cdot|^{s+1}  \prod_{l=2}^{2p-1} f_l \Vert_{L^{\infty}}\lesssim \Vert f_1 \Vert_{\dot{H}^{s,1}_{\omega}} \prod_{l=2}^{2p-1} \Vert f_l \Vert_{X_{s,4}^{\omega}}.
\end{align*}
For the last term we have
\begin{align*}
\Vert |\cdot|^2 (\Omega_{ij} f_1) (\partial_m f_2) \prod_{l=3}^{2p-1} f_l \Vert_{L^2} &\leq \Vert|\cdot|^{-(s-1)} \partial_m f_2 \Vert_{L^2} \Vert |\cdot|^{s+1} (\Omega_{ij} f_1) \prod_{l=3}^{2p-1} f_l \Vert_{L^{\infty}}\\
&\lesssim \Vert f_2 \Vert_{\dot{H}^s} \left( \Vert \Omega_{ij} f_1 \Vert_{L^{\infty}(B_1)} \prod_{l=3}^{2p-1} \Vert f_l \Vert_{L^{\infty}(B_1)}\right)\\
&+ \Vert f_2 \Vert_{\dot{H}^s} \left( \Vert |\cdot|^{3} (\Omega_{ij} f_1) \prod_{l=3}^{2p-1} f_l \Vert_{L^{\infty}(B_1^c)} \right).
\end{align*}
Again, we can exploit the $L^{\infty}$-embedding and find
\begin{align*}
 \Vert \Omega_{ij} f_1 \Vert_{L^{\infty}(B_1)} \prod_{l=3}^{2p-1} \Vert f_l \Vert_{L^{\infty}(B_1)} \lesssim \Vert f_1 \Vert_{X_{s,4}^{\omega}} \prod_{l=3}^{2p-1} \Vert f_l \Vert_{X_{s,4}^{\omega}}.
\end{align*}
For the remaining part we do the same as in \eqref{3} where we arrange the product in such a way that $\Omega_{ij} f_1$ gets measured in $L^{\infty}(\R^3)$ without a weight
\begin{align*}
\Vert |\cdot|^{3} (\Omega_{ij} f_1) \prod_{l=3}^{2p-1} f_l \Vert_{L^{\infty}(B_1^c)} &\leq \Vert \Omega_{ij} f_1 \Vert_{L^{\infty}(B_1^c)} \prod_{l=3}^{\frac{3p+3}{2}} \Vert |\cdot|^{\frac{2}{p-1}-\varepsilon} f_l  \Vert_{L^{\infty}(B_1^{c})} \prod_{n=\frac{3p+5}{2}}^{2p-1} \Vert f_n \Vert_{L^{\infty}(B_1^{c})}\\
&\lesssim \Vert f_1 \Vert_{X_{s,4}^{\omega}} \prod_{l=3}^{2p-1} \Vert f_l \Vert_{X_{s,4}^{\omega}},
\end{align*}
where the product $\prod_{n=\frac{3p+5}{2}}^{2p-1} \Vert f_n \Vert_{L^{\infty}(B_1^{c})}$ only occurs if $p \geq 7$.\\
\\
For $|\alpha| =4$ the following terms arise
\begin{align} \label{Subcases}
\partial^{\beta} \left( (\Omega_{ij} f_1) \prod_{l=2}^{2p-1} f_l \right) , \ |\cdot| \partial^{\gamma} \left( (\Omega_{ij} f_1) \prod_{l=2}^{2p-1} f_l \right) \ \text{and} \ |\cdot|^2 \partial^{\sigma} \left( (\Omega_{ij} f_1) \prod_{l=2}^{2p-1} f_l \right),
\end{align}
where $\beta, \gamma, \sigma \in \N^3_0$ with $|\beta|=2$, $|\gamma|=3$ and $|\sigma|=4$. The first term can be estimated by
\begin{align*}
\Vert \partial^{\beta} \left( (\Omega_{ij} f_1) \prod_{l=2}^{2p-1} f_l \right) \Vert_{L^2} &  \leq \Vert \Omega_{ij} f_1 \prod_{l=2}^{2p-1} f_l \Vert_{\dot{H}^2}\lesssim \Vert \Omega_{ij} f_1 \prod_{l=2}^{2p-1} f_l \Vert_{\dot{H}^s \cap \dot{H}^4}\\
&\lesssim \Vert \Omega_{ij} f_1 \Vert_{\dot{H}^s \cap \dot{H}^4} \prod_{l=2}^{2p-1} \Vert f_l \Vert_{\dot{H}^s \cap \dot{H}^4} \lesssim \prod_{l=1}^{2p-1} \Vert f_l \Vert_{X_{s,4}^{\omega}}.
\end{align*}
For the second term let $\gamma_1, \gamma_2 \in \N_0^3$ with $|\gamma_1|+|\gamma_2| =|\gamma|=3$. By  the Leibniz rule we have to find suitable estimates in $L^2(\R^3)$ for 
\begin{align*}
|\cdot| (\partial^{\gamma_1} \Omega_{ij} f_1) \partial^{\gamma_2} \left( \prod_{l=2}^{2p-1} f_l \right).
\end{align*}
To do so we need to distinguish three possible choices of $\gamma_1$ and $\gamma_2$. If $|\gamma_1| \in \{ 2,3 \}$ and $|\gamma_2| \in \{ 0,1 \}$ we have
\begin{align*}
\Vert |\cdot| (\partial^{\gamma_1} \Omega_{ij} f_1) \partial^{\gamma_2} \left( \prod_{l=2}^{2p-1} f_l \right) \Vert_{L^2} &\leq \Vert \partial^{\gamma_1} \Omega_{ij} f_1 \Vert_{L^2} \Vert |\cdot| (\partial^{\gamma_2} f_2) \prod_{l=3}^{2p-1} f_l \Vert_{L^{\infty}},
\end{align*}
where $\Vert \partial^{\gamma_1} \Omega_{ij} f_1 \Vert_{L^2} \lesssim \Vert \Omega_{ij} f_1 \Vert_{\dot{H}^{|\gamma_1|}} \lesssim \Vert \Omega_{ij} f_1 \Vert_{\dot{H}^{s} \cap \dot{H}^4} \lesssim \Vert f_1 \Vert_{X_{s,4}^{\omega}}$ and
\begin{align*}
\Vert |\cdot| (\partial^{\gamma_2} f_2) \prod_{l=3}^{2p-1} f_l \Vert_{L^{\infty}} &\leq \Vert (\partial^{\gamma_2} f_2) \prod_{l=3}^{2p-1} f_l \Vert_{L^{\infty}(B_1)} +  \Vert |\cdot|^3 (\partial^{\gamma_2} f_2) \prod_{l=3}^{2p-1} f_l \Vert_{L^{\infty}(B_1^c)}.
\end{align*}
We estimate the terms in $L^{\infty}(B_1)$ by
\begin{align*}
\Vert (\partial^{\gamma_2} f_2) \prod_{l=3}^{2p-1} f_l \Vert_{L^{\infty}(B_1)} &\leq \Vert \partial^{\gamma_2} f_2 \Vert_{\dot{H}^{s-|\gamma_2|} \cap \dot{H}^{4-|\gamma_2|}} \prod_{l=3}^{2p-1} \Vert f_l \Vert_{X_{s,4}^{\omega}}\lesssim \prod_{l=2}^{2p-1} \Vert f_l \Vert_{X_{s,4}^{\omega}}.
\end{align*}
To estimate the weighted product in $L^{\infty}(B_1^c)$ we do the same as in \eqref{3} which is possible due to $|\gamma_2| \leq 1$.\\
The next subcase arises if $|\gamma_1|=1$ and $|\gamma_2| =2$. Again, we apply the Leibniz rule for $\partial^{\gamma_2}$ with multiindices $\gamma_2^1, \gamma_2^2 \in \N_0^3$ such that $|\gamma_2^1|+|\gamma_2^2|=|\gamma_2|$. By this we get
\begin{align*}
&\Vert |\cdot|^{-(s-1)} \partial^{\gamma_1} \Omega_{ij} f_1 \Vert_{L^{2}} \Vert \partial^{\gamma_2^1} f_2 \Vert_{L^{\infty}} \Vert \partial^{\gamma_2^2} f_3 \Vert_{L^{\infty}} \Vert |\cdot|^s \prod_{l=4}^{2p-1} f_l \Vert_{L^{\infty}} \lesssim \prod_{l=1}^{2p-1} \Vert f_l \Vert_{X_{s,4}^{\omega}},
\end{align*}
where we used that 
\begin{align} \label{4}
\Vert |\cdot|^s \prod_{l=4}^{2p-1} f_l \Vert_{L^{\infty}} &\lesssim 
\prod_{l=4}^{2p-1} \Vert f_l \Vert_{L^{\infty}(B_1)} + \Vert \left( |\cdot|^{\frac{2}{p-1}-\varepsilon} \right)^{p-1} |\cdot|^{(p-1) \varepsilon} \prod_{l=4}^{2p-1} f_l \Vert_{L^{\infty}(B_1^c)} \nonumber \\
&\lesssim \prod_{l=4}^{2p-1} \Vert f_l \Vert_{X_{s,4}^{\omega}} + \prod_{l=4}^{p+3} \Vert |\cdot|^{\frac{2}{p-1}-\varepsilon} f_l \Vert_{L^{\infty}(B_1^c)} \prod_{n=p+4}^{2p-1} \Vert f_l \Vert_{L^{\infty}(B_1^c)} \nonumber \\
&\lesssim \prod_{l=4}^{2p-1} \Vert f_l \Vert_{X_{s,4}^{\omega}},
\end{align}
by exploiting $\frac{p-1}{2}\varepsilon \leq \frac{2}{p-1}-\varepsilon$. If $|\gamma_1|=0$ and $|\gamma_2|=3$ we use the Leibniz rule for the partial derivative $\partial^{\gamma_2}$ with multiindices $\gamma_2^1, \gamma_2^2, \gamma_2^3 \in \N_0^3$ such that $|\gamma_2^1|+|\gamma_2^2|+|\gamma_2^3|=|\gamma_2|$. The worst case appears if $|\gamma_2^1|=|\gamma_2^2|=|\gamma_2^3|=1$, so it suffices to treat this case. We split the term as
\begin{align*}
&\Vert \Omega_{ij} f_1 \Vert_{L^{\infty}} \Vert |\cdot|^{-(s-1)} \partial^{\gamma_2^1} f_2 \Vert_{L^2} \Vert |\cdot|^s (\partial^{\gamma_2^2} f_3) (\partial^{\gamma_2^3} f_4) \prod_{l=5}^{2p-1} f_l \Vert_{L^{\infty}}\\ &\lesssim \Vert f_1 \Vert_{X_{s,4}^{\omega}} \Vert f_2 \Vert_{\dot{H}^s} \prod_{l=3}^{2p-1} \Vert f_l \Vert_{X_{s,4}^{\omega}}\lesssim \prod_{l=1}^{2p-1} \Vert f_l \Vert_{X_{s,4}^{\omega}},
\end{align*}
where we used the same proof as in \eqref{4} for the product appearing here.
\\
The last case of \eqref{Subcases} is the most challenging as it contains the strongest weight and the highest order of derivatives. Again, we use the Leibniz rule with $|\sigma|= |\sigma_1|+|\sigma_2|$, $\sigma_1, \sigma_2 \in \N_0^3$ such that the product splits into products of the form
\begin{align*}
|\cdot|^2 (\partial^{\sigma_1} \Omega_{ij} f_1) \partial^{\sigma_2} \left( \prod_{l=2}^{2p-1} f_l \right).
\end{align*}
We distinguish the cases $|\sigma_1| \geq 2$ and $|\sigma_1| \in \{ 0,1 \}$. If $|\sigma_1| \geq 2$ and $|\sigma_2| \leq 2$ we use the Leibniz rule for $\partial^{\sigma_2}$ with multiindices $\sigma_2^1, \sigma_2^2 \in \N_0^3$ such that $|\sigma_2^1|+|\sigma_2^2|=|\sigma_2|$. With that at hand we split the term in $L^2(\R^3)$ in the following way
\begin{align*}
&\Vert \partial^{\sigma_1} \Omega_{ij} f_1 \Vert_{L^2}  \Vert \partial^{\sigma_2^1} f_2 \Vert_{L^{\infty}} \Vert \partial^{\sigma_2^2} f_3 \Vert_{L^{\infty}} \Vert |\cdot|^2 \prod_{l=4}^{2p-1} f_l \Vert_{L^{\infty}}\\
&\lesssim \Vert \Omega_{ij} f_1 \Vert_{\dot{H}^{|\sigma_1|}} \Vert f_2 \Vert_{X_{s,4}^{\omega}} \Vert f_3 \Vert_{X_{s,4}^{\omega}} \left( \prod_{l=4}^{2p-1} \Vert f_l \Vert_{L^{\infty}(B_1)} + \Vert |\cdot|^2 \prod_{l=4}^{2p-1} f_l \Vert_{L^{\infty}(B_1^c)} \right),
\end{align*}
where we use a similar argument as in \eqref{4}. Namely,
\begin{align} \label{5}
\begin{split}
\Vert |\cdot|^2  & \prod_{l=4}^{2p-1} f_l \Vert_{L^{\infty}(B_1^c)} = \Vert \left( |\cdot|^{\frac{2}{p-1}-\varepsilon} \right)^{p-1} |\cdot|^{(p-1)\varepsilon} \prod_{l=4}^{2p-1} f_l \Vert_{L^{\infty}(B_1^c)}  \\
&\lesssim \prod_{l=4}^{p+3} \Vert |\cdot|^{\frac{2}{p-1}-\varepsilon} f_l \Vert_{L^{\infty}(B_1^c)} \prod_{n=p+4}^{2p-1} \Vert f_n \Vert_{L^{\infty}(B_1^c)} \lesssim \prod_{l=4}^{2p-1} \Vert f_l \Vert_{X_{s,4}^{\omega}}
\end{split}
\end{align}
where we exploited $(p-1)\varepsilon \leq \frac{2}{p-1}-\varepsilon$. The next possible choice of $\sigma_1$ and $\sigma_2$ is given by $|\sigma_1| =1$ and $|\sigma_2|=3$. Again, we apply the Leibniz rule to the partial derivative $\partial^{\sigma_2}$ with multiindices $\sigma_2^1, \sigma_2^2, \sigma_2^3 \in \N_0^3$ such that $|\sigma_2^1|+|\sigma_2^2|+|\sigma_2^3| =|\sigma_2|$. Let $\max_{n = 1,2,3} |\sigma_2^n| = |\sigma_2^1|$. If $|\sigma_2^1| \geq 2$ we have
\begin{align*}
&\Vert \partial^{\sigma_2^1} f_2 \Vert_{L^2} \Vert \partial^{\sigma_1} \Omega_{ij} f_1 \Vert_{L^{\infty}} \Vert |\cdot|^2 (\partial^{\sigma_2^2} f_3)(\partial^{\sigma_2^3} f_4) \prod_{l=5}^{2p-1} f_l \Vert_{L^{\infty}} \\
& \lesssim \Vert f_2 \Vert_{\dot{H}^{|\sigma_2^1|}} \Vert \Omega_{ij} f_1 \Vert_{\dot{H}^{2} \cap \dot{H}^{3}} \prod_{l=3}^{2p-1} \Vert f_l \Vert_{X_{s,4}^{\omega}} \lesssim \prod_{l=1}^{2p-1} \Vert f_l \Vert_{X_{s,4}^{\omega}},
\end{align*}
where we adapted the proof of \eqref{5} to this product which is possible since $|\sigma_2^2| \leq 1$ and $|\sigma_2^3| \leq 1$. If $|\sigma_2^1| = |\sigma_2^2| = |\sigma_2^3| =1$ we do the following
\begin{align*}
&\Vert |\cdot|^{-(s-1)} \partial^{\sigma_1} \Omega_{ij} f_1 \Vert_{L^2} \Vert |\cdot|^{s+1} (\partial^{\sigma_2^1} f_2) (\partial^{\sigma_2^2} f_3) (\partial^{\sigma_2^3} f_4) \prod_{l=5}^{2p-1} f_l \Vert_{L^{\infty}}\\
&\lesssim \Vert \Omega_{ij} f_1 \Vert_{\dot{H}^s}\\
&\cdot \left( \Vert (\partial^{\sigma_2^1} f_2) (\partial^{\sigma_2^2} f_3) (\partial^{\sigma_2^3} f_4) \prod_{l=5}^{2p-1} f_l \Vert_{L^{\infty}(B_1)} + \Vert |\cdot|^3 (\partial^{\sigma_2^1} f_2) (\partial^{\sigma_2^2} f_3) (\partial^{\sigma_2^3} f_4) \prod_{l=5}^{2p-1} f_l \Vert_{L^{\infty}(B_1^c)} \right)\\
&\lesssim \prod_{l=1}^{2p-1} \Vert f_l \Vert_{X_{s,4}^{\omega}},
\end{align*}
where we used the same proof as in \eqref{3} which is possible due to the order of $\sigma_2^1,\sigma_2^2,\sigma_2^3$.\\
The remaining case of possible $\sigma_1$ and $\sigma_2$ arises if $|\sigma_1|=0$ and $|\sigma_2|=4$. To treat this case we use the Leibniz rule for $\partial^{\sigma_2}$ and consider two possible subcases. The first one appears if the partial derivative $\partial^{\sigma_2}$ distributes on at most three functions, i.e.~there exist three multiindices $\sigma_2^1, \sigma_2^2, \sigma_2^3 \in \N_0^3$ such that $|\sigma_2^1|+|\sigma_2^2|+|\sigma_2^3| =|\sigma_2|$ and $\max_{n =1,2,3} |\sigma_2^n| = |\sigma_2^1| \geq 2$. For that we split the term in $L^2(\R^3)$ in the following way
\begin{align*}
\Vert \partial^{\sigma_2^1} f_2 \Vert_{L^2} \Vert \Omega_{ij} f_1 \Vert_{L^{\infty}} \Vert |\cdot|^2 (\partial^{\sigma_2^2} f_3)(\partial^{\sigma_2^3} f_4) \prod_{l=5}^{2p-1} f_l \Vert_{L^{\infty}}  & \lesssim \Vert f_2 \Vert_{\dot{H}^{|\sigma_2^1|}} \Vert \Omega_{ij} f_1 \Vert_{\dot{H}^{s} \cap \dot{H}^{4}} \prod_{l=3}^{2p-1} \Vert f_l \Vert_{X_{s,4}^{\omega}}\\
&\lesssim \prod_{l=1}^{2p-1} \Vert f_l \Vert_{X_{s,4}^{\omega}},
\end{align*}
where the estimate for the weighted product in $L^{\infty}(\R^3)$ follows from the fact that $|\sigma_2^2| \leq 1$ and $|\sigma_2^3| \leq 1$. The other case arises if four partial derivatives $\partial_m$, $\partial_u$, $\partial_v$ and $\partial_w$ with $m,u,v,w \in \{ 1,2,3 \}$ of order one distribute on four different functions. To treat this case we split the term in $L^2(\R^3)$ into
\begin{align*}
&\Vert |\cdot|^{-(s-1)} \partial_m f_2 \Vert_{L^{\infty}} \Vert \Omega_{ij} f_1 \Vert_{L^{\infty}} \Vert |\cdot|^{s+1} \partial_u f_3 \partial_v f_4 \partial_w f_5 \prod_{l=6}^{2p-1} f_l \Vert_{L^{\infty}}\\
&\lesssim \Vert f_2 \Vert_{\dot{H}^s} \Vert f_1 \Vert_{X_{s,4}^{\omega}}\\ &\cdot \left( \Vert (\partial_u f_3) (\partial_v f_4) (\partial_w f_5) \prod_{l=6}^{2p-1} f_l \Vert_{L^{\infty}(B_1)} + \Vert |\cdot|^{3} (\partial_u f_3) (\partial_v f_4) (\partial_w f_5) \prod_{l=6}^{2p-1} f_l \Vert_{L^{\infty}(B_1^c)} \right)\\
&\lesssim \prod_{l=1}^{2p-1} \Vert f_l \Vert_{X_{s,4}^{\omega}},
\end{align*}
where we again used the proof of \eqref{3} for the weighted product appearing here.\\
\\
Let $p=3$, then we have $s=\frac{1}{2}+\varepsilon <1$, $\lfloor s \rfloor =0$ and $k=2$. Hence, we have to estimate terms in $L^2(\R^3)$ of the form
\begin{align*}
|\cdot|^2 (\Omega_{ij} f_1) \prod_{l=2}^{5} f_l \quad \text{and} \quad \partial^{\alpha} \left( |\cdot|^2 (\Omega_{ij}f_1) \prod_{l=2}^{5} f_l \right),
\end{align*}
for $\alpha \in \N_0^3$ with $|\alpha|=2$ and $i,j \in \{ 1,2,3 \}$, $i<j$. To treat the first term we use the splitting
\begin{align*}
\Vert |\cdot|^{-s} \Omega_{ij} f_1 \Vert_{L^2} \Vert |\cdot|^{s+2} \prod_{l=2}^{5} f_l \Vert_{L^{\infty}} &\lesssim \Vert f_1 \Vert_{\dot{H}^{s,1}_{\omega}} \left( \Vert \prod_{l=2}^{5} f_l \Vert_{L^{\infty}(B_1)} + \Vert |\cdot|^{3} \prod_{l=2}^{5} f_l \Vert_{L^{\infty}(B_1^c)} \right)\\
&\lesssim \Vert f_1 \Vert_{X_{s,2}^{\omega}} \left( \prod_{l=2}^{5} \Vert f_l \Vert_{X_{s,2}^{\omega}} + \prod_{l=2}^{5} \Vert |\cdot|^{1-\varepsilon} f_l \Vert_{L^{\infty}(B_1^c)} \right)\\
&\lesssim \prod_{l=1}^5 \Vert f_l \Vert_{X_{s,2}^{\omega}},
\end{align*}
where we exploited $3 \varepsilon < 1- \varepsilon$. For the second term we apply the Leibniz rule which gives the following subcases in $L^2(\R^3)$
\begin{align*}
(\Omega_{ij}f_1) \prod_{l=2}^{5} f_l , \quad |\cdot| (\partial^{\beta_1} \Omega_{ij} f_1) \partial^{\beta_2} \left( \prod_{l=2}^5 f_l \right) \quad \text{and} \quad |\cdot|^2 (\partial^{\alpha_1} \Omega_{ij} f_1) \partial^{\alpha_2} \left( \prod_{l=2}^5 f_l \right),
\end{align*}
where $\alpha_1, \alpha_2, \beta_1, \beta_2 \in \N_0^3$ with $|\alpha_1|+|\alpha_2|=2$ and $|\beta_1|+|\beta_2|=1$. For the first subcase we have
\begin{align*}
\Vert |\cdot|^{-s} \Omega_{ij} f_1 \Vert_{L^2} \Vert |\cdot|^s \prod_{l=2}^{5} f_l \Vert_{L^{\infty}} &\lesssim \Vert f_1 \Vert_{X_{s,2}^{\omega}} \left(  \Vert \prod_{l=2}^{5} f_l \Vert_{L^{\infty}(B_1)} + \Vert |\cdot| \prod_{l=2}^{5} f_l \Vert_{L^{\infty}(B_1^c)} \right)\lesssim \prod_{l=1}^5 \Vert f_l \Vert_{X_{s,2}^{\omega}},
\end{align*}
by the same argument as for the first term. For the next subcase we have the two options $|\beta_1|=1$, $|\beta_2| =0$ or $|\beta_1|=0$, $|\beta_2| = 1$. Let $|\beta_1|=1$ and $|\beta_2|=0$, then we split the term in the following way
\begin{align*}
\Vert \partial^{\beta_1} \Omega_{ij} f_1 \Vert_{L^2} \Vert |\cdot| \prod_{l=2}^5 f_l \Vert_{L^{\infty}} & \lesssim \Vert f_1 \Vert_{\dot{H}^{1,1}_{\omega}} \left(  \Vert \prod_{l=2}^{5} f_l \Vert_{L^{\infty}(B_1)} + \Vert |\cdot| \prod_{l=2}^{5} f_l \Vert_{L^{\infty}(B_1^c)} \right)\lesssim \prod_{l=1}^5 \Vert f_l \Vert_{X_{s,2}^{\omega}}.
\end{align*}
If $|\beta_1|=0$ and $|\beta_2|=1$ we get
\begin{align*}
\Vert \partial^{\beta_2} f_2 \Vert_{L^2} \Vert |\cdot| (\Omega_{ij} f_1) \prod_{l=3}^5 f_l \Vert_{L^{\infty}}  & \lesssim \Vert f_2 \Vert_{\dot{H}^{1}} \left(  \Vert (\Omega_{ij} f_1) \prod_{l=3}^{5} f_l \Vert_{L^{\infty}(B_1)} + \Vert |\cdot| (\Omega_{ij} f_1) \prod_{l=3}^{5} f_l \Vert_{L^{\infty}(B_1^c)} \right)\\&  \lesssim \prod_{l=1}^5 \Vert f_l \Vert_{X_{s,2}^{\omega}},
\end{align*}
where we used
\begin{align*}
\Vert |\cdot| (\Omega_{ij} f_1) \prod_{l=3}^{5} f_l \Vert_{L^{\infty}(B_1^c)} &\lesssim \Vert |\cdot|^2 (\Omega_{ij} f_1) \prod_{l=3}^{5} f_l \Vert_{L^{\infty}(B_1^c)} \lesssim \Vert \Omega_{ij} f_1 \Vert_{L^{\infty}} \prod_{l=3}^{5} \Vert |\cdot|^{1-\varepsilon} f_l \Vert_{L^{\infty}(B_1^c)}\\
&\lesssim \Vert f_1 \Vert_{X_{s,2}^{\omega}} \prod_{l=3}^{5} \Vert f_l \Vert_{X_{s,2}^{\omega}}.
\end{align*}
To handle the last subcase we distinguish two different choices of $\alpha_1$ and $\alpha_2$. First, let $|\alpha_1| \geq 1$ and $|\alpha_2| = 2 - |\alpha_1| \leq 1$, then we find
\begin{align*}
&\Vert \partial^{\alpha_1} \Omega_{ij} f_1 \Vert_{L^2} \Vert |\cdot|^2 (\partial^{\alpha_2} f_2) \prod_{l=3}^5 f_l \Vert_{L^{\infty}}\\ 
&\lesssim \Vert f_1 \Vert_{\dot{H}^{|\alpha_1|,1}_{\omega}} \Vert |\cdot|^{1-\varepsilon} \partial^{\alpha_2} f_2 \Vert_{L^{\infty}} \Vert |\cdot|^{1+\varepsilon} \prod_{l=3}^5 f_l \Vert_{L^{\infty}}\\
&\lesssim \prod_{n=1}^2 \Vert f_n \Vert_{X_{s,2}^{\omega}} \left( \Vert \prod_{l=3}^5 f_l \Vert_{L^{\infty}(B_1)} + \prod_{l=3}^4 \Vert |\cdot|^{1-\varepsilon} f_l \Vert_{L^{\infty}(B_1^c)} \Vert f_5 \Vert_{L^{\infty}(B_1^c)} \right)\\
&\lesssim \prod_{l=1}^5 \Vert f_l \Vert_{X_{s,2}^{\omega}}.
\end{align*}
For $|\alpha_1|=0$ and $|\alpha_2|=2$ we apply the Leibniz rule for the differential operator $\partial^{\alpha_2}$ with $\alpha_2^1, \alpha_2^2 \in \N_0^3$ such that $|\alpha_2^1|+|\alpha_2^2|=|\alpha_2|$. Let $\max_{n=1,2} |\alpha_2^n| = |\alpha_2^1|$, then we have $|\alpha_2^2| \leq 1$ and we find
\begin{align*}
&\Vert \partial^{\alpha_2^1} f_2 \Vert_{L^2} \Vert |\cdot|^2 (\Omega_{ij} f_1)(\partial^{\alpha_2^2} f_3) \prod_{l=4}^5 f_l \Vert_{L^{\infty}}\\ 
&\lesssim \Vert f_2 \Vert_{\dot{H}^{|\alpha_2^1|}} \Vert \Omega_{ij} f_1 \Vert_{L^{\infty}} \Vert |\cdot|^{1-\varepsilon} \partial^{\alpha_2^2} f_3 \Vert_{L^{\infty}} \left( \prod_{l=4}^5 \Vert f_l \Vert_{L^{\infty}(B_1)} + \prod_{l=4}^5 \Vert |\cdot|^{1-\varepsilon} f_l \Vert_{L^{\infty}(B_1^c)} \right) \lesssim \prod_{l=1}^5 \Vert f_l \Vert_{X_{s,2}^{\omega}}.
\end{align*}
\end{proof}
As a consequence, we obtain local Lipschitz continuity of the operator $\mc N$ from \eqref{N}.
\begin{lemma} \label{Lipschitz}
The nonlinearity $\mathcal{N}$ from \eqref{N}, initially defined on $C_c^\infty(\R^3)$,  extends to a continuous map $\mathcal{N}: X_{s,k}^{\omega} \rightarrow X_{s,k}^{\omega}$ satisfying
\begin{align} \label{Nonlinearity}
\Vert \mathcal{N}(f) - \mathcal{N}(g) \Vert_{X_{s,k}^{\omega}} \leq \gamma ( \Vert f \Vert_{X_{s,k}^{\omega}}, \Vert g \Vert_{X_{s,k}^{\omega}})(\Vert f \Vert_{X_{s,k}^{\omega}} + \Vert g \Vert_{X_{s,k}^{\omega}}) \Vert f-g \Vert_{X_{s,k}^{\omega}},
\end{align}
for all $f,g \in X_{s,k}^{\omega}$, where $\gamma: [0, \infty) \times [0, \infty) \rightarrow [0, \infty)$ is a continuous function.
\end{lemma}

\begin{proof}
For better readability we skip the variable $y$ in the following. By \eqref{Rewrite} we find
\begin{align*}
&\mathcal{N}(f)-\mathcal{N}(g) = \sum_{n=2}^p \binom{p}{n} \phi^{p-n} (f^n -g^n) -c |\cdot|^2 \sum_{n=2}^{2p-1} \binom{2p-1}{n} \phi^{2p-1-n}(f^n-g^n)\\
&=(f-g) \left( \sum_{n=2}^p \binom{p}{n} \phi^{p-n} \sum_{j=0}^{n-1} f^{n-1-j}g^j - c |\cdot|^2 \sum_{n=2}^{2p-1} \binom{2p-1}{n} \phi^{2p-1-n} \sum_{j=0}^{n-1} f^{n-1-j}g^j \right).
\end{align*}
Applying Lemma \ref{WProductLemma} and the Banach algebra property to every single term appearing in these two sums \eqref{Nonlinearity} follows.
\end{proof}

\subsection{Construction of strong solutions} We consider Duhamel's formula 
\begin{align} \label{Duhamelsformula}
\varphi(\tau) = S^X(\tau) v_0 + \int_0^{\tau} S^X(\tau-\tau^{\prime}) \mathcal{N}(\varphi(\tau^{\prime}))d \tau^{\prime}, \quad \tau \geq 0,
\end{align}
in order to construct a strong solution to the initial value problem \eqref{CE} with $T=1$.
We define the Banach space
\begin{align*}
\mathcal{X}_{s,k}^{\omega} := \{ \varphi \in C([0,\infty),X_{s,k}^{\omega}) : \Vert \varphi \Vert_{\mathcal{X}_{s,k}^{\omega}} := \sup_{\tau \geq 0} e^{\nu \tau} \Vert \varphi(\tau) \Vert_{X_{s,k}^{\omega}} < \infty \},
\end{align*}
with $\nu >0$ being the constant from \eqref{Semigroupdecay}. Furthermore, we set \[\mathcal{X}_{s,k}^{\omega}(\delta) := \{ \varphi \in \mathcal{X}_{s,k}^{\omega} : \Vert \varphi \Vert_{\mathcal{X}_{s,k}^{\omega}} \leq \delta \}\]
for $\delta >0$. To run a fixed point argument, we use a Lyapunov-Perron type argument and define for each unstable eigenvalue $\lambda_j$, $j \in \{1,...,\tilde{N} \}$ a correction term
\begin{align}\label{Def:Corr1}
C_j(\varphi,u) := P_j u + P_j \int_0^{\infty} e^{- \lambda_j \tau^{\prime}} \mathcal{N}(\varphi(\tau^{\prime})) d \tau^{\prime},
\end{align}
and set
\begin{align}\label{Def:Corr2}
C(\varphi,u) := \sum_{j=1}^{\tilde{N}} C_j(\varphi,u),
\end{align}
for $\varphi \in \mathcal{X}_{s,k}^{\omega}$ and $u \in X_{s,k}^{\omega}$.  With this at hand we are able to define the operator 
\begin{align*}
[K(\varphi,u)](\tau) := S^X(\tau)(u- C(\varphi,u)) + \int_0^{\tau} S^X(\tau-\tau^{\prime}) \mathcal{N}(\varphi(\tau^{\prime}))d \tau^{\prime}, \quad \tau \geq 0.
\end{align*}
As a first step to find a strong solution to the problem \eqref{CE}, we run a fixed point argument for the operator $K(\cdot,u)$ for suitably small $u \in X_{s,k}^{\omega}$. Then, we determine a co-dimension $N$ Lipschitz manifold $\mathcal{M}$ of functions $u \in X_{s,k}^{\omega}$ for which the correction term $C(\varphi,u)$ vanishes. In summary, we find for every $u \in \mathcal{M}$ a unique fixed point $\varphi = \varphi(u)$ that solves the fixed point equation $K(\varphi,u) = \varphi$ with $C(\varphi,u) =0$, i.e. a solution to Duhamel's formula \eqref{Duhamelsformula}. The precise statements are given in the following.

\begin{lemma} \label{OpK}
For all sufficiently small $\delta >0$, all sufficiently large $c>0$ and $u \in X_{s,k}^{\omega}$ satisfying $\Vert u \Vert_{X_{s,k}^{\omega}} \leq \frac{\delta}{c}$, the operator $K(\cdot, u)$ maps the ball $\mathcal{X}_{s,k}^{\omega}(\delta)$ into itself and satisfies
\begin{align*}
\Vert K(\varphi_1,u)-K(\varphi_2,u) \Vert_{\mathcal{X}_{s,k}^{\omega}} \leq \frac{1}{2} \Vert \varphi_1 - \varphi_2 \Vert_{\mathcal{X}_{s,k}^{\omega}},
\end{align*}
for all $\varphi_1, \varphi_2 \in \mathcal{X}_{s,k}^{\omega}(\delta)$.
\end{lemma}

\begin{proof}
Let $u \in X_{s,k}^{\omega}$ with $\Vert u \Vert_{X_{s,k}^{\omega}} \leq \frac{\delta}{c}$ and $\varphi \in \mathcal{X}_{s,k}^{\omega}(\delta)$, then we have
\begin{align*}
[K(\varphi, u)](\tau) &= (1-P)S^X(\tau)u + \int_0^{\tau} (1-P)S^X(\tau-\tau^{\prime}) \mathcal{N}(\varphi(\tau^{\prime}))d \tau^{\prime}\\ 
&- \sum_{j=1}^{\tilde{N}} \int_{\tau}^{\infty} e^{\lambda_j (\tau- \tau^{\prime})} P_j \mathcal{N}(\varphi(\tau^{\prime}))d \tau^{\prime}.
\end{align*}
The exponential decay of the semigroup on the stable subspace by Proposition \ref{Semigroupdec} together with estimate \eqref{Nonlinearity} for the nonlinearity implies
\begin{align*}
\Vert [K(\varphi,u)](\tau)\Vert_{X_{s,k}^{\omega}} &\lesssim e^{-\nu \tau} \frac{\delta}{c} +\delta^2 e^{-\nu \tau} \int_0^{\tau} e^{-\nu \tau^{\prime}} d \tau^{\prime} + \sum_{j=1}^{\tilde{N}} \delta^2 e^{\lambda_j \tau} \int_{\tau}^{\infty} e^{-(\lambda_j + 2 \nu)\tau^{\prime}} d \tau^{\prime}\\
&\lesssim e^{-\nu \tau} \left( \frac{\delta}{c} +\frac{\delta^2}{\nu} \right).
\end{align*}

For sufficiently small $\delta >0$ and sufficiently large $c >0$ we infer $\Vert K(\varphi,u) \Vert_{\mathcal{X}_{s,k}^{\omega}} \leq \delta$. Next, let $u \in X_{s,k}^{\omega}$ with $\Vert u \Vert_{X_{s,k}^{\omega}} \leq \frac{\delta}{c}$ and $\varphi_1, \varphi_2 \in \mathcal{X}_{s,k}^{\omega}(\delta)$, then we find
\begin{align*}
\Vert [K(\varphi_1,u)](\tau) - [K(\varphi_2,u)](\tau) \Vert_{X_{s,k}^{\omega}} &\lesssim \int_0^{\infty} e^{-\nu(\tau-\tau^{\prime})} \Vert \mathcal{N}(\varphi_1(\tau^{\prime}))- \mathcal{N}(\varphi_2(\tau^{\prime})) \Vert_{X_{s,k}^{\omega}} d \tau^{\prime}\\ 
&+ \sum_{j=1}^{\tilde{N}} \int_{\tau}^{\infty} e^{\lambda_j (\tau-\tau^{\prime})} \Vert \mathcal{N}(\varphi_1(\tau^{\prime}))- \mathcal{N}(\varphi_2(\tau^{\prime})) \Vert_{X_{s,k}^{\omega}} d \tau^{\prime}\\
&\lesssim \delta e^{-\nu \tau} \left( \frac{1}{\nu} + \sum_{j=1}^{\tilde{N}} \frac{e^{-\nu \tau}}{\lambda_j+2 \nu} \right) \Vert \varphi_1 - \varphi_2 \Vert_{\mathcal{X}_{s,k}^{\omega}}.
\end{align*}
Again, for sufficiently small $\delta >0$ we find $\Vert K(\varphi_1,u)-K(\varphi_2,u) \Vert_{\mathcal{X}_{s,k}^{\omega}} \leq \frac{1}{2} \Vert \varphi_1 - \varphi_2 \Vert_{\mathcal{X}_{s,k}^{\omega}}$.
\end{proof}

With this result at hand the contraction mapping principle yields for every $u \in X_{s,k}^{\omega}$ with $\Vert u \Vert_{X_{s,k}^{\omega}} \leq \frac{\delta}{c}$ a unique fixed point $\varphi_u \in \mathcal{X}_{s,k}^{\omega}(\delta)$ of the equation $K(\varphi,u) = \varphi$. In order to obtain a solution to the original problem, we construct a Lipschitz manifold determining the vanishing of the correction term. In the following let
\begin{align*}
B_{\frac{\delta}{c}} := \{ u \in X_{s,k}^{\omega} : \Vert u \Vert_{X_{s,k}^{\omega}} \leq \frac{\delta}{c} \},
\end{align*}
and 
\begin{align*}
C(u):=C(\varphi_u,u),
\end{align*}
with $\varphi_u \in \mathcal{X}_{s,k}^{\omega}(\delta)$ being the unique fixed point of the operator $K(\cdot,u)$ associated to $u \in B_{\frac{\delta}{c}}$. Before  constructing the Lipschitz manifold, we need the following Lemma concerning Lipschitz regularity of the map $u \mapsto \varphi_u$.

\begin{lemma} \label{Lipschitzcont}
Let $u \in B_{\frac{\delta}{c}}$, then for sufficiently small $\delta >0$ the map $u \mapsto \varphi_u = K(\varphi,u)$ is Lipschitz continuous from $X_{s,k}^{\omega}$ to $\mathcal{X}_{s,k}^{\omega}$.
\end{lemma}

\begin{proof}
Let $u, \tilde{u} \in B_{\frac{\delta}{c}}$ and $\varphi_u, \varphi_{\tilde{u}} \in  \mathcal{X}_{s,k}^{\omega}(\delta)$ be the unique fixed points of the operator $K$. Let $\tau > 0$, then we have
\begin{align*}
\Vert \varphi_u(\tau) - \varphi_{\tilde{u}}(\tau) \Vert_{X_{s,k}^{\omega}} &\leq e^{-\nu \tau} \left( \Vert u-\tilde{u} \Vert_{X_{s,k}^{\omega}} + \Vert C(\tilde{u})-C(u) \Vert_{X_{s,k}^{\omega}} \right)\\
&+ \int_0^{\tau} e^{-\nu (\tau - \tau^{\prime})} \delta e^{-2 \nu \tau^{\prime}} d \tau^{\prime} \Vert \varphi_u - \varphi_{\tilde{u}} \Vert_{\mathcal{X}_{s,k}^{\omega}} \\
&\lesssim e^{-\nu \tau} \left( \Vert u - \tilde{u} \Vert_{X_{s,k}^{\omega}} + \delta \Vert \varphi_u - \varphi_{\tilde{u}} \Vert_{\mathcal{X}_{s,k}^{\omega}} \right),
\end{align*}
where we used that
\begin{align*}
\Vert C(\tilde{u})-C(u) \Vert_{X_{s,k}^{\omega}} \lesssim \Vert u - \tilde{u} \Vert_{X_{s,k}^{\omega}} + \delta \Vert \varphi_u - \varphi_{\tilde{u}} \Vert_{\mathcal{X}_{s,k}^{\omega}}.
\end{align*}
This implies $\Vert \varphi_u - \varphi_{\tilde{u}} \Vert_{\mathcal{X}_{s,k}^{\omega}} \lesssim \Vert u - \tilde{u} \Vert_{X_{s,k}^{\omega}}$ for sufficiently small $\delta > 0$.
\end{proof}

\begin{proposition}\label{manifoldM}
	For all sufficiently small $\delta > 0$ and all sufficiently large $c>0$, 
	there exists a co-dimension $N$ Lipschitz manifold 
	$
	\mathcal{M} = \mathcal{M}_{\delta,c} \subset X_{s,k}^{\omega}
	$
	with the following properties:
	
	\begin{enumerate}[i)]
		\item $0 \in \mathcal{M}$.
		
		\item There exists a Lipschitz map 
		$
		W : \ker(P) \cap B_{\frac{\delta}{2c}} \to \rg(P)
		$
		such that
		\[
		\mathcal{M} = 
		\left\{ v + W(v) : v \in \ker(P),\,
		\|v\|_{X_{s,k}^{\omega}} \le \tfrac{\delta}{2c} \right\} \subset
		\left\{ u \in B_{\frac{\delta}{c}} : C(u)=0 \right\}.
		\]
		
		\item For every $u \in \mathcal{M}$ there exists a unique 
		$\varphi_u \in \mathcal{X}_{s,k}^{\omega}(\delta)$ satisfying
		\begin{align}\label{classicalsolutioneq}
			\varphi(\tau) 
			= S^X(\tau)u 
			+ \int_0^{\tau} S^X(\tau-\tau') 
			\mathcal{N}(\varphi(\tau'))\, d\tau',
			\quad \tau \ge 0 .
		\end{align}
		
		\item There exists a constant $\tilde{c} > 2c$ such that if
		$
		u \in B_{\frac{\delta}{\tilde{c}}}
	$ and $C(u)=0$,
		then $u \in \mathcal{M}$.
	\end{enumerate}
\end{proposition}
\noindent 

\begin{proof}
We write $X_{s,k}^{\omega} = \ker(P) \oplus \rg(P)$, such that each $u \in X_{s,k}^{\omega}$ splits into $u=v+w$ with $v \in \ker(P)$ and $w \in \rg(P)$. Let $v \in \ker(P) \cap B_{\frac{\delta}{2c}}$ and $\tilde{B}_{\frac{\delta}{c}}(v) := \{ w \in \rg(P) : \Vert v +w \Vert_{X_{s,k}^{\omega}} \leq \frac{\delta}{c} \}$. To each $v \in \ker(P) \cap B_{\frac{\delta}{2c}}$ we define a map
\begin{align*}
C_v : \tilde{B}_{\frac{\delta}{c}}(v) \rightarrow \rg(P), \quad w \mapsto C_v(w) := C(v+w).
\end{align*}
We show that $C_v$ is invertible at $0$ for sufficiently small $v$. Note that 
\begin{align*}
C(v+w) = 0 \quad \Leftrightarrow \quad w + \sum_{j=1}^{\tilde{N}} P_j \int_0^{\infty} e^{- \lambda_j \tau^{\prime}} \mathcal{N}(\varphi_{v+w}(\tau^{\prime})) d \tau^{\prime} = 0.
\end{align*}
To run a fixed point argument, we define $\Omega_v(w) := -\sum_{j=1}^{\tilde{N}} P_j \int_0^{\infty} e^{- \lambda_j \tau^{\prime}} \mathcal{N}(\varphi_{v+w}(\tau^{\prime})) d \tau^{\prime}$ and show that for $v \in \ker(P)$ with $\Vert v \Vert_{X_{s,k}^{\omega}} \leq \frac{\delta}{2c}$ the map $\Omega_v$ maps $\tilde{B}_{\frac{\delta}{c}}(v)$ into itself and is a contraction. Let $v \in \ker(P)$ with $\Vert v \Vert_{X_{s,k}^{\omega}} \leq \frac{\delta}{2c}$ and $w \in \tilde{B}_{\frac{\delta}{c}}(v)$, then
\begin{align*}
\Vert v + \Omega_v(w) \Vert_{X_{s,k}^{\omega}} &\lesssim \frac{\delta}{2c} + \sum_{j=1}^{\tilde{N}} \int_0^{\infty} \delta^2 e^{-(\lambda_j + 2 \nu) \tau^{\prime}} d \tau^{\prime} \lesssim \frac{\delta}{2c} + \delta^2,
\end{align*}
implying $\Omega_v(w) \in \tilde{B}_{\frac{\delta}{c}}(v)$ for sufficiently small $\delta >0$. Next, let $w, \tilde{w} \in \tilde{B}_{\frac{\delta}{c}}(v)$, then we have
\begin{align*}
\Vert \Omega_v(w) - \Omega_v(\tilde{w}) \Vert_{X_{s,k}^{\omega}} &\lesssim \sum_{j=1}^{\tilde{N}} \int_0^{\infty} \delta e^{- (\lambda_j + 2 \nu) \tau^{\prime}} d \tau^{\prime} \Vert \varphi_{v+w} - \varphi_{v+\tilde{w}} \Vert_{\mathcal{X}_{s,k}^{\omega}}\\
&\lesssim \delta \Vert w - \tilde{w} \Vert_{X_{s,k}^{\omega}},
\end{align*}
by the Lipschitz continuity of the map $u \mapsto \varphi_u$ from Lemma \ref{Lipschitzcont}. Hence, for sufficiently small $\delta >0$ the map $\Omega_v$ is a contraction on $\tilde{B}_{\frac{\delta}{c}}(v)$. Banach's fixed point theorem yields for every $v \in \ker(P)$ with $\Vert v \Vert_{X_{s,k}^{\omega}} \leq \frac{\delta}{2c}$ a unique $w \in \tilde{B}_{\frac{\delta}{c}}(v)$ satisfying $C_v(w) = C(v+w)=0$. This defines a map 
\begin{align*}
W : \ker(P) \cap B_{\frac{\delta}{2c}} \rightarrow \rg(P), \quad v \mapsto W(v) := C_v^{-1}(0).
\end{align*}
To show Lipschitz continuity of the map $W$, let $v, \tilde{v} \in \ker(P) \cap B_{\frac{\delta}{2c}}$ and $w \in \tilde{B}_{\frac{\delta}{c}}(v)$, $\tilde{w} \in \tilde{B}_{\frac{\delta}{c}}(\tilde{v})$ be the unique fixed points of $\Omega_v$ and $\Omega_{\tilde{v}}$, respectively. By the Lipschitz continuity of the nonlinearity together with the Lipschitz continuity of the map $u \mapsto \varphi_u$ we find
\begin{align*}
\Vert W(v) - W(\tilde{v}) \Vert_{X_{s,k}^{\omega}} &= \Vert w - \tilde{w} \Vert_{X_{s,k}^{\omega}}\\ 
&\lesssim \sum_{j=1}^{\tilde{N}} \int_0^{\infty} e^{- \lambda_j \tau^{\prime}} \Vert \mathcal{N}(\varphi_{v+w}(\tau^{\prime}))- \mathcal{N}(\varphi_{\tilde{v}+\tilde{w}} (\tau^{\prime})) \Vert_{X_{s,k}^{\omega}} d \tau^{\prime}\\
&\lesssim \delta \Vert \varphi_{v+w} - \varphi_{\tilde{v}+\tilde{w}} \Vert_{\mathcal{X}_{s,k}^{\omega}}\\
&\lesssim \delta \left( \Vert v-\tilde{v} \Vert_{X_{s,k}^{\omega}} + \Vert w - \tilde{w} \Vert_{X_{s,k}^{\omega}} \right),
\end{align*}
which implies $\Vert W(v) - W(\tilde{v}) \Vert_{X_{s,k}^{\omega}} \lesssim \Vert v-\tilde{v} \Vert_{X_{s,k}^{\omega}}$ for sufficiently small chosen $\delta >0$.\\
\\
Obviously, $0 \in \ker(P) \cap B_{\frac{\delta}{2c}}$ and since $w = 0 \in \rg(P)$ satisfies $\Omega_0(0)=0$, we infer $0 \in \mathcal{M}$.\\
Now, let $u \in X_{s,k}^{\omega}$ satisfy $C(u)=0$ and let $v_u := (1-P)u \in \ker(P)$, then $\Vert v_u \Vert_{X_{s,k}^{\omega}} \leq \frac{\delta}{2c}$ if $\Vert u \Vert_{X_{s,k}^{\omega}} \leq \frac{\delta}{\tilde{c}}$ for $\tilde{c} >0$ sufficiently large. Furthermore, $W(v_u) = Pu$ by uniqueness, which shows $u \in \mathcal{M}$.
\end{proof}

With this result at hand we are able to construct classical solutions to the Cauchy problem \eqref{CE} with $T=1$ and initial data belonging to the manifold $\mathcal{M}$.

\begin{proposition} \label{SolutionC1}
There exists $\nu > 0$ such that for any $\varphi_0 \in \mathcal{M} \cap \mathcal{D}(\mathcal{L}^X)$ the initial value problem 
\begin{align*}
\begin{cases} \partial_{\tau} \varphi = L \varphi (\tau, \cdot) + \mathcal{N}(\varphi(\tau, \cdot)), \quad \tau >0, \\ \varphi(0, \cdot) = \varphi_0, \end{cases}
\end{align*}
has a unique solution $\varphi \in C([0, \infty), \mathcal{D}(\mathcal{L}^X)) \cap C^1([0, \infty), X_{s,k}^{\omega})$ satisfying $\Vert \varphi(\tau) \Vert \lesssim e^{-\nu \tau}$ for $\tau \geq 0$.
\end{proposition}

\begin{proof}
Let $\delta, c >0$ be such that $\mathcal{M}$ is well-defined by Proposition \ref{manifoldM}. Let $\varphi_0 \in \mathcal{M} \cap \mathcal{D}(\mathcal{L}^X)$, then $\Vert \varphi_0 \Vert_{X_{s,k}^{\omega}} \leq \frac{\delta}{c}$ and $C(\varphi_0) =0$. Furthermore, there exists a unique $\varphi \in C([0, \infty), X_{s,k}^{\omega})$ satisfying equation \eqref{classicalsolutioneq} with $u = \varphi_0$ for all $\tau \geq 0$.\\
Since $\varphi_0 \in \mathcal{D}(\mathcal{L}^X)$ the local Lipschitz continuity of $\mathcal{N}$ implies that $\varphi \in C([0, \infty), \mathcal{D}(\mathcal{L}^X)) \cap C^1([0, \infty), X_{s,k}^{\omega})$ is indeed a classical solution to the above operator equation (see, e.g.~\cite{MR1691574}, p.~60, Proposition 4.3.9).

The existence of $\nu >0$ and the corresponding exponential decay of the solution follow from the fact that $\varphi \in \mathcal{X}_{s,k}^{\omega}(\delta)$, i.e.~$\Vert \varphi(\tau) \Vert_{X_{s,k}^{\omega}} \leq \delta e^{- \nu \tau}$ with $\nu > 0$ being the constant from \eqref{Semigroupdecay}.
\end{proof}

The last step to prove Theorem \ref{maintheorem} is given by the following Lemma showing the existence of suitable parameters for the initial data.

\begin{lemma} \label{Correctionvanish}
 For every sufficiently large $M >0$ there exists a sufficiently small $\delta >0$ such that the following holds. Let $v_0 \in X_{s,k}^{\omega}$ satisfy $\Vert v_0 \Vert_{X_{s,k}^{\omega}} \leq \frac{\delta}{M^2}$, then there exist parameters $a_{i,j} \in [-\frac{\delta}{M}, \frac{\delta}{M} ]$, $(i,j) \in I$ with $I := \{(i,j) : 1 \leq i \leq N_j : 1 \leq j \leq \tilde{N} \}$ such that 
\begin{align*}
C\Big(v_0 + \sum_{(i,j) \in I} a_{i,j} f_j^{i}\Big) =0.
\end{align*}
In particular, we have $v_0 + \sum_{(i,j) \in I} a_{i,j} f_j^{i} \in \mathcal{M}$. Furthermore, the parameters $a_{i,j} = a_{i,j}(v_0)$ depend Lipschitz continuously on the data, i.e.
\begin{align*}
\sum_{(i,j) \in I} | a_{i,j}(v_0) - a_{i,j}(\tilde{v}_0) | \lesssim \Vert v_0 - \tilde{v}_0 \Vert_{X_{s,k}^{\omega}},
\end{align*}
for all $v_0, \tilde{v}_0 \in X_{s,k}^{\omega}$ satisfying $\Vert v_0 \Vert_{X_{s,k}^{\omega}} \leq \frac{\delta}{M^2}$ and $\Vert \tilde{v}_0 \Vert_{X_{s,k}^{\omega}} \leq \frac{\delta}{M^2}$.
\end{lemma}

\begin{proof}
First, we observe that $\Vert v_0 + \sum_{(i,j) \in I} a_{i,j} f_j^{i} \Vert_{X_{s,k}^{\omega}} \leq \frac{\delta}{\tilde{c}} \leq \frac{\delta}{c}$ with $\delta, \tilde{c}, c > 0$ from Proposition \ref{manifoldM} for all $v_0 \in X_{s,k}^{\omega}$ with $\Vert v_0 \Vert_{X_{s,k}^{\omega}} \leq \frac{\delta}{M^2}$ and $a_{i,j} \in [-\frac{\delta}{M}, \frac{\delta}{M} ]$, $(i,j) \in I$ if $M >0$ is chosen sufficiently large. In the following, we write $a = (a_{i,j})_{(i,j) \in I}$ and denote by $\varphi_{v_0,a} \in \mathcal{X}_{s,k}^{\omega}(\delta)$ the unique fixed point of the operator $K$ associated with $v_0+ \Sigma_{(i,j)\in I}a_{i,j}f_j^i$. To find suitable parameters that guarantee the vanishing of the correction term, we run a fixed point argument. To do so, we refine each correction term $C_j$, $j \in \{1,...,\tilde{N} \}$ by 
\begin{align*}
C_j(\varphi, u) = \sum_{i=1}^{N_j} C_j^{i}(\varphi, u),
\end{align*}
with
\begin{align*}
C_j^{i}(\varphi, u) := P_j^{i} u + P_j^{i} \int_0^{\infty} e^{- \lambda_j \tau^{\prime}} \mathcal{N}(\varphi(\tau^{\prime})) d \tau^{\prime},
\end{align*}
for $\varphi \in \mathcal{X}_{s,k}^{\omega}$ and $u \in X_{s,k}^{\omega}$. Then $C(v_0 + \sum_{(i,j) \in I} a_{i,j} f_j^{i}) =0$ if and only if $C_j^{i}(v_0 + \sum_{(i,j) \in I} a_{i,j} f_j^{i}) =0$ for all $(i,j) \in I$. By definition of the correction terms via the projections $P_j^{i}$ vanishing of $C_j^{i}(v_0 + \sum_{(i,j) \in I} a_{i,j} f_j^{i})$ is equivalent to
\begin{align*}
\langle v_0 + \sum_{(i,j) \in I} a_{i,j} f_j^{i} + \int_0^{\infty} e^{-\lambda_j \tau^{\prime}} \mathcal{N}(\varphi_{v_0,a}(\tau^{\prime})) d \tau^{\prime}, f_j^{i} \rangle_{L^2_{\sigma}} = 0,
\end{align*}
which again is equivalent to 
\begin{align*}
\langle v_0, f_j^{i} \rangle_{L^2_{\sigma}} + \langle \int_0^{\infty} e^{-\lambda_j \tau^{\prime}} \mathcal{N}(\varphi_{v_0,a}(\tau^{\prime})) d \tau^{\prime}, f_j^{i} \rangle_{L^2_{\sigma}} + a_{i,j} =0.
\end{align*}
We define a continuous map $F^{v_0} : B_{\frac{\delta}{M}}(\R^N) \rightarrow B_{\frac{\delta}{M}}(\R^N)$ with $N = \sum_{j=1}^{\tilde{N}} \sum_{i=1}^{N_j}$ by 
\begin{align*}
F_{ij}^{v_0}(a) = - \langle v_0, f_j^{i} \rangle_{L^2_{\sigma}} - \langle \int_0^{\infty} e^{-\lambda_j \tau^{\prime}} \mathcal{N}(\varphi_{v_0,a}(\tau^{\prime})) d \tau^{\prime}, f_j^{i} \rangle_{L^2_{\sigma}},
\end{align*}
for $(i,j) \in I$ and $a = (a_{ij})_{(i,j) \in I} \in B_{\frac{\delta}{M}}(\R^N)$, which is well-defined by the continuous embedding of $X_{s,k}^{\omega}$ into $L^2_{\sigma}(\R^3)$ (see Lemma \ref{Embedding2}), the local Lipschitz continuity of $\mathcal{N}$ (see Lemma \ref{Lipschitz})  and the Lipschitz continuity of the map $u \mapsto \varphi_u$ (see Lemma \ref{Lipschitzcont}). Furthermore, we have
\begin{align*}
|F^{v_0}(a)| = \sum_{(i,j) \in I} |F^{v_0}_{ij}(a)| \lesssim \frac{\delta}{M^2} + \delta^2, 
\end{align*}
implying that $F^{v_0}(a) \in B_{\frac{\delta}{M}}(\R^N)$ for $a \in B_{\frac{\delta}{M}}(\R^N)$ provided that $\delta >0$ is sufficiently small. Next, we show that $F^{v_0}$ is a contraction. Let $a, \tilde{a} \in B_{\frac{\delta}{M}}(\R^N)$, then 
\begin{align*}
|F^{v_0}_{ij}(a) - F^{v_0}_{ij}(\tilde{a})| &\lesssim \delta \Vert \varphi_{v_0,a} - \varphi_{v_0, \tilde{a}} \Vert_{\mathcal{X}_{s,k}^{\omega}} \lesssim \delta \sum_{(i,j) \in I} |a_{ij} - \tilde{a}_{ij} |,
\end{align*}
which shows 
\begin{align*}
\sum_{(i,j) \in I} |F^{v_0}_{ij}(a) - F^{v_0}_{ij}(\tilde{a})| &\leq \frac{1}{2} \sum_{(i,j) \in I} |a_{ij} - \tilde{a}_{ij} |,
\end{align*}
if $\delta > 0$ is chosen sufficiently small. By the contraction mapping principle we infer the existence of a unique fixed point $a \in B_{\frac{\delta}{M}}(\R^N)$ of $F^{v_0}(a) = a$. We denote this fixed point by $a(v_0) = (a_{i,j}(v_0))_{(i,j) \in I}$, which by construction satisfies $C(v_0 + \sum_{(i,j) \in I} a_{i,j} f_j^{i})=0$.\\
\\
For the Lipschitz continuous dependency of the parameters on the initial data let $v_0, \tilde{v}_0 \in X_{s,k}^{\omega}$ satisfy $\Vert v_0 \Vert_{X_{s,k}^{\omega}}, \Vert \tilde{v}_0 \Vert_{X_{s,k}^{\omega}} \leq \frac{\delta}{M^2}$ and let $a = a(v_0)$, $\tilde{a} = {a}(\tilde{v}_0)$ be the corresponding unique fixed points of $F^{v_0}$ and $F^{\tilde{v}_0}$ respectively. Then we find
\begin{align*}
|a_{i,j}-\tilde{a}_{i,j}| &\lesssim | \langle v_0 - \tilde{v}_0 , f_{j}^{i} \rangle_{L^2_{\sigma}} | + \int_0^{\infty} e^{-\lambda_j \tau^{\prime}} \Vert \mathcal{N}(\varphi_{v_0,a}(\tau^{\prime})) - \mathcal{N}(\varphi_{\tilde{v}_0,\tilde{a}}(\tau^{\prime})) \Vert_{X_{s,k}^{\omega}} d \tau^{\prime}\\
&\lesssim \Vert v_0 - \tilde{v}_0 \Vert_{X_{s,k}^{\omega}} + \delta \left( \Vert v_0 - \tilde{v}_0 \Vert_{X_{s,k}^{\omega}} + \sum_{(i,j) \in I} |a_{i,j}-\tilde{a}_{i,j}| \right),
\end{align*}
implying that
\begin{align*}
\sum_{(i,j) \in I} |a_{i,j}-\tilde{a}_{i,j}| \lesssim \Vert v_0 - \tilde{v}_0 \Vert_{X_{s,k}^{\omega}},
\end{align*}
provided $\delta >0$ is chosen sufficiently small.
\end{proof}

\subsection{Proof of Theorem \ref{maintheorem}} Let $\gamma := \frac{\delta}{M^2}$ with $\delta >0$ and $M>0$ such that Lemma \ref{Correctionvanish} holds. Let $v_0 \in \mc S(\R^3)$ satisfy $\Vert v_0 \Vert_{X_{s,k}^{\omega}} \leq \gamma$, then there exist parameters $a_{i,j} = a_{i,j}(v_0) \in [-\frac{\delta}{M}, \frac{\delta}{M} ]$, $(i,j) \in I$ with $I := \{(i,j) : 1 \leq i \leq N_j : 1 \leq j \leq \tilde{N} \}$ depending Lipschitz continuously on $v_0$ with respect to $X_{s,k}^{\omega}$ and a solution $\varphi \in C([0, \infty), \mathcal{D}(\mathcal{L}^X)) \cap C^1([0, \infty), X_{s,k}^{\omega})$ to the problem
\begin{align} \label{EQ1}
\begin{cases} \partial_{\tau} \varphi = L \varphi (\tau, \cdot) + \mathcal{N}(\varphi(\tau, \cdot)), \\ \varphi(0, \cdot) = v_0 + \sum_{(i,j) \in I} a_{i,j} f_j^{i},
\end{cases}
\end{align}
by Lemma \ref{Correctionvanish} and Proposition \ref{SolutionC1}. To establish higher regularity, we first observe that $\varphi$ solves \eqref{EQ1} pointwise by the embedding of $X_{s,k}^{\omega}$ into $L^{\infty}(\R^3)$, see Lemma \ref{LemmaEmbedding}. Next, we exploit that $\mathcal{L}^X$ is a bounded perturbation of $\mathcal{L}_0^X$ and use Duhamel's formula for the free semigroup to find
\begin{align*}
\varphi(\tau) = S_0^X(\tau) U(v_0) + \int_0^{\tau} S_0^X(\tau - \tau^{\prime}) \left( \mc L_1 \varphi(\tau^{\prime}) + \mathcal{N}(\varphi(\tau^{\prime} )) \right) d \tau^{\prime}, \quad \text{for $\tau \geq 0$},
\end{align*}
where $U(v_0) := v_0 + \sum_{(i,j) \in I} a_{i,j} f_j^{i}$. Smoothing of the free semigroup (see Appendix, A.4 in \cite{MR4730409}) implies $\varphi \in X_{s,\tilde{k}}^{\omega}$ for all $\tau \geq 0$ and $\tilde{k} \geq k$. The continuous embedding $X_{s,\tilde{k}}^{\omega} \hookrightarrow C^{m}(\R^3)$ with $\tilde{k} \geq \frac{3}{2} +m$ shows $\varphi(\tau) \in C^{\infty}(\R^3)$ for all $\tau \geq 0$. A generalized version of Schwarz's theorem (see \cite{MR385023}, p.~235, Theorem 9.41) allows the exchange of $\mathcal{L}^X$ and $\partial_{\tau}$. Furthermore, since we are dealing with an odd power nonlinearity, differentiation of the right hand side of equation \eqref{EQ1} with respect to $\tau$ gives higher regularity also in $\tau$. Hence, mixed derivatives of all orders in $\tau$ and the spatial variable $y$ exist, whichs shows $\varphi \in C^{\infty}([0, \infty) \times \R^3)$.\\
Again by Proposition \ref{SolutionC1}, $\varphi$ satisfies $\Vert \varphi(\tau) \Vert \lesssim e^{-\nu \tau}$ for $\tau \geq 0$. Translating this back into physical coordinates $x$ and $t$ via
\begin{align*}
u(t,x) = (1-t)^{-\frac{1}{p-1}} \left( \phi \left( \frac{|x|}{\sqrt{1-t}} \right) + \varphi \left( \log \left( \frac{1}{1-t} \right), \frac{x}{\sqrt{1-t}} \right) \right),
\end{align*}
gives the solution $u$ with $\tilde{\varphi}(t, \cdot) := \varphi \left( \log \left( \frac{1}{1-t} \right), \cdot \right)$ for $t \in [0,1)$ as claimed in Theorem \ref{maintheorem}. Furthermore, we have
\begin{align*}
\Vert \tilde{\varphi}(t, \cdot) \Vert_{X_{s,k}^{\omega}} = \Vert \varphi(\tau, \cdot) \Vert_{X_{s,k}^{\omega}} \leq \delta e^{-\nu \tau} = \delta (1-t)^{\nu},
\end{align*}
which shows \eqref{Conv} by the continuous embedding of $X_{s,k}^{\omega} \hookrightarrow L^{\infty}(\R^3)$ by Lemma \ref{LemmaEmbedding}.

\subsection{Stable blowup}

Let $p=3$ and $c \in (\frac{1}{4}, \frac{1}{3})$. From Lemma \ref{K0} we infer that the linear operator possesses only one positive eigenvalue given by $\lambda =1$ and the corresponding eigenspace is one-dimensional and spanned by the function $g(x) = (b+|x|^2)^{-\frac{p}{p-1}}$. In particular, the correction term defined in \eqref{Def:Corr1}-\eqref{Def:Corr2} consists only of a single term. Now, instead of correcting the initial data as in Lemma \ref{Correctionvanish}, we exploit the origin of this instability and adjust the blowup time $T$ accordingly. In particular, we consider Eq.~\eqref{CE} with initial data $\mathcal{U}(v_0, T)$ defined in \eqref{InitialU} for $T > 0$.

\begin{proposition}\label{Prop:StableBlowup}
There is a large enough $M > 0$ and a small enough $\delta > 0$ such that the following holds. For all real-valued  $v_0 \in X_{s,k}^{\omega}$ with $\Vert v_0 \Vert_{X_{s,k}^{\omega}} \leq \frac{\delta}{M^2}$, there exists a $T = T(v_0) \in [1-\frac{\delta}{M}, 1 + \frac{\delta}{M}]$  and a unique $\varphi \in C([0, \infty), X_{s,k}^{\omega})$ that satisfies 
\begin{align}\label{Eq:DuhU}
\varphi(\tau) = S^X(\tau)  \mathcal{U}(v_0, T(v_0)) + \int_0^{\tau} S^X(\tau-\tau^{\prime}) \mathcal{N}(\varphi(\tau^{\prime}))d \tau^{\prime}, \quad \tau \geq 0
\end{align}
for all $\tau \geq 0$ such that 
\[ \Vert \varphi(\tau) \Vert_{X_{s,k}^{\omega}} \leq \delta e^{- \nu \tau}, \quad \forall \tau \geq 0.  \]
\end{proposition}

\begin{proof}
The proof follows a basic scheme detailed in e.g.~\cite{MR4730409}. First, one can easily show that for any $0 < \delta \leq \frac{1}{2}$ and for every $v_0 \in X_{s,k}^{\omega}$ the map 
\begin{align*}
T \mapsto \mathcal{U}(v_0, T) : [1-\delta, 1+ \delta] \rightarrow X_{s,k}^{\omega},
\end{align*}
is continuous. Furthermore, 
\begin{align*}
\Vert \mathcal{U}(v_0, T) \Vert_{X_{s,k}^{\omega}} \lesssim \Vert v_0 \Vert_{X_{s,k}^{\omega}} + |T-1|,
\end{align*}
for all $v_0 \in X_{s,k}^{\omega}$ and $T \in [\frac{1}{2},\frac{3}{2}]$, see the proof of Lemma 5.2 in \cite{MR4730409}.

Now, by choosing $M$ sufficiently large,  for any  $v_0 \in X_{s,k}^{\omega}$ with $\Vert v_0 \Vert_{X_{s,k}^{\omega}} \leq \frac{\delta}{M^2}$ and any $T \in  [1-\frac{\delta}{M}, 1 + \frac{\delta}{M}]$ we can achieve $\| \mathcal{U}(v_0, T) \|_{X_{s,k}^{\omega}} \leq \frac{\delta}{c}$, where $c$ is the constant from Lemma \ref{OpK}. Consequently, for $\delta$ sufficiently small, for every such pair $(v_0,T)$ there exists a unique solution $\varphi_{(v_0,T)}$ in $\mathcal{X}_{s,k}^{\omega}(\delta)$ satisfying $K(\varphi_{(v_0,T)},\mathcal{U}(v_0, T) ) = \varphi_{(v_0,T)}$. Furthermore, the map $T \mapsto \varphi_{(v_0,T)}$ is continuous.

The crucial step is to show that the correction term contained in $K$ vanishes for fixed $v_0$ and a suitable choice of $T=T(v_0)$. This is based on a Taylor expansion of the term $T^{\frac{1}{p-1}} \phi( \sqrt{T} \cdot) - \phi(\cdot)$ in $T=1$ appearing in the initial data operator $\mathcal{U}(v_0, T)$. Namely, we have
\begin{align*}
\mathcal{U}(v_0,T) = T^{\frac{1}{p-1}} v_0(\sqrt{T} \cdot) + \tilde{C} (1-T) g(\cdot) + (1-T)^2 R(T, \cdot),
\end{align*}
for some non-zero constant $\tilde{C} \in \R$. The remainder satisfies $\Vert R(T, \cdot) \Vert_{X_{s,k}^{\omega}} \lesssim 1$ for all $T$ close to $1$. Since the range of the correction term is spanned by the eigenfunction $g$, it suffices to show that the projection onto $\mathrm{span}(g)$ vanishes. This, together with the above expansion, leads to an equation of the form
\begin{align}\label{Eq:FPT}
T = F(T) + 1 
\end{align}
for a continuous function $F$ satisfying $|F(T)| \lesssim \frac{\delta}{M^2} + \delta^2$. We refer the reader to the proof of Theorem 5.4 in \cite{MR4730409} for the details. By choosing $M$ sufficiently large and $\delta$ sufficiently small, an application of the intermediate value theorem implies the existence of a solution to \eqref{Eq:FPT}. Hence, the corresponding $\varphi_{(v_0,T(v_0))}$ satisfies \eqref{Eq:DuhU}. 

The uniqueness of the solution in $C([0, \infty), X_{s,k}^{\omega})$ follows from standard arguments, see again \cite{MR4730409}, Theorem 5.4. 
\end{proof}

\subsection{Proof of Theorem \ref{maintheorem2}} Starting with Proposition \ref{Prop:StableBlowup} one argues as in Proposition \ref{SolutionC1} to obtain a unique solution of the operator equation \eqref{CE}. The rest of the argument is analogous to the proof of Theorem \ref{maintheorem} modulo the form of the initial data.

\appendix
\section{Bound on the number of negative eigenvalues} \label{GGMTv}

Let $A$ be a linear operator defined on $\mathcal{D}(A) = C^{\infty}_c(\R^+)$ and given by
\begin{align*}
[Au](r) := -u^{\prime \prime}(r) + \frac{a}{r^2} u(r) +  b r^2 u(r) +  V(r)u(r),
\end{align*}
for $a > -\frac{1}{4}$, $a \neq 0$, $b >0$ and $V \in C[0,\infty)$ a bounded, real-valued potential. We denote the self-adjoint closure of  $A$ by  $\mathcal{A}$.
To formulate the theorem, we write the operator as 
\begin{align*}
[Au](r) := -u^{\prime \prime}(r) + \frac{\alpha}{r^2} u(r) + Q(r)u(r),
\end{align*}
for $- \frac14 < \alpha < a$ and $Q(r) := \frac{(a - \alpha)}{r^2} + b r^2 + V(r)$. 
 
 \begin{theorem} \label{TheoremA}
Let $Q = Q_{+} - Q_{-}$ where $Q_{\pm} \geq 0$ and let $\kappa \geq \frac{3}{2}$. Then the number of negative eigenvalues $\nu(\mathcal{A})$ of $\mathcal{A}$ is bounded by
 \begin{align} \label{Boundoneigenvalues}
 \nu(\mathcal{A}) \leq \frac{(\kappa-1)^{\kappa-1} \Gamma(2 \kappa)}{(4 \alpha +1)^{\kappa - \frac{1}{2}} \kappa^{\kappa} \Gamma(\kappa)^2} \int_0^{\infty} r^{2 \kappa -1} Q_{-}(r)^{\kappa} dr.
 \end{align}
 \end{theorem}

 \begin{proof}
First, the fact that $\mc A$ is bounded from below and has compact resolvent implies that the spectrum of $\mc A$ in the left half-plane consists of finitely many real eigenvalues. Let $u_0$ denote the regular solution of $Au = 0$, i.e.~the one which is in $L^2(0,c)$ for some $c >0$. We note the assumption on $a$ implies $u_0(r) \sim r^{\gamma}$ around zero for some $\gamma > \frac{1}{2}$. From oscillation theory, see e.g.~\cite{Weidmann1987, GesSimTes1996} it follows that the number of eigenvalues $\nu(\mathcal{A})$ is given by the number of zeros of $u_0$ in $(0,\infty)$.  Let $\{r_0,...,r_{\nu(\mathcal{A})} \}$ be the ordered set of zeros of $u_0$ with $r_0 =0$. Then, for each pair $(r_k,r_{k+1})$ with $k \in \{0,...,\nu(\mathcal{A})-1 \}$ we find 
 \begin{align*}
 0 &= \int_{r_k}^{r_{k+1}} u_0^{\prime}(r)^2 + \frac{\alpha}{r^2} u_0(r)^2 + Q(r) u_0(r)^2 dr \\
 &\geq \int_{r_k}^{r_{k+1}} u_0^{\prime}(r)^2 + \frac{\alpha}{r^2} u_0(r)^2 dr - \int_{r_k}^{r_{k+1}} Q_{-}(r) u_0(r)^2 dr \\
  &\geq \int_{r_k}^{r_{k+1}} u_0^{\prime}(r)^2 + \frac{\alpha}{r^2} u_0(r)^2 dr - \Vert r^{\frac{2 \kappa -1}{\kappa}} Q_{-} \Vert_{L^{\kappa}([r_k, r_{k+1}])} \Vert r^{-\frac{2 \kappa -1}{\kappa}} u_0^2 \Vert_{L^{\frac{\kappa}{\kappa-1}}([r_k, r_{k+1}])},
 \end{align*}
where we used H\"older's inequality with $\kappa \geq 1$. We note that $ \Vert r^{-\frac{2 \kappa -1}{\kappa}} u_0^2 \Vert_{L^{\frac{\kappa}{\kappa-1}}([0, r_1])} < \infty$. Since $u_0 \neq 0$ in $(r_k, r_{k+1})$ we get for each $k$,
 \begin{align*}
 \int_{r_k}^{r_{k+1}} r^{2 \kappa -1} Q_{-}(r)^{\kappa} dr &\geq  \frac{\left( \int_{r_k}^{r_{k+1}} u_0^{\prime}(r)^2 + \frac{\alpha}{r^2} u_0(r)^2 dr \right)^{\kappa}}{\left( \int_{r_k}^{r_{k+1}} r^{-\frac{2 \kappa -1}{\kappa -1}} |u_0(r)|^{\frac{2 \kappa}{\kappa -1}} dr \right)^{\kappa -1}}\\
 &\geq \left( \inf_{u \in H^1(\R^+) \setminus \{0 \}} F_{\alpha}(u) \right)^{\kappa},
 \end{align*}
where we define the functional $F_{\alpha}$ for $u \in H^1(\R^+) \setminus \{ 0 \}$ by
 \begin{align*}
 F_{\alpha}(u) := \frac{\int_{0}^{\infty} u^{\prime}(r)^2 + \frac{\alpha}{r^2} u(r)^2 dr}{\left( \int_{0}^{\infty} r^{-\frac{2 \kappa -1}{\kappa -1}} |u(r)|^{\frac{2\kappa}{\kappa -1}} dr \right)^{\frac{\kappa -1}{\kappa}}}.
 \end{align*}
and extending $u_0$ by zero outside of $[r_k, r_{k+1}]$.
 By the change of variables $r = e^{\frac{z}{\sqrt{4 \alpha +1}}}$ and $\frac{u(r)}{ \sqrt{r}} = w(z)$ we infer that
 \begin{align*}
 F_{\alpha}(u) = (4 \alpha +1)^{\frac{2 \kappa -1}{2 \kappa}} F(w),
 \end{align*}
 where
 \begin{align*}
 F(w) = \frac{\int_{- \infty}^{\infty} w^{\prime}(z)^2 + \frac{1}{4} w(z)^2 dz}{\left( \int_{- \infty}^{\infty} |w(z)|^{\frac{2 \kappa}{\kappa -1}} dz \right)^{\frac{\kappa -1}{\kappa}}}.
 \end{align*}
 From \cite{CreDonSchSne2017}, Lemma A.2 we know 
 \begin{align*}
 \inf_{w \in H^1(\R) \setminus \{0\}} F(w) = \frac{\kappa}{\kappa -1} \left( \frac{(\kappa -1) \Gamma(\kappa)^2}{\Gamma(2 \kappa)} \right)^{\frac{1}{\kappa}},
 \end{align*}
 which implies
 \begin{align*}
  \int_{0}^{\infty} r^{2 \kappa -1} Q_{-}(r)^{\kappa} dr &=   \sum_{k=0}^{\nu(\mathcal{A})-1}  \int_{r_k}^{r_{k+1}} r^{2 \kappa -1} Q_{-}(r)^{\kappa} dr +  \int_{r_{\nu(\mathcal{A})}}^{\infty} r^{2 \kappa -1} Q_{-}(r)^{\kappa} dr\\
  &\geq \nu(\mathcal{A}) (4 \alpha +1)^{\frac{2 \kappa -1}{2}} \frac{\kappa^{\kappa} \Gamma(\kappa)^2}{(\kappa -1)^{\kappa -1} \Gamma(2 \kappa)},
 \end{align*}
 leading to \eqref{Boundoneigenvalues}.
 \end{proof}
 
 \section{Discussion on unstable eigenvalues} \label{Discussion}
In the following we discuss the number of unstable eigenvalues of the linear operator $\mathcal{L} : \mathcal{D}(\mathcal{L}) \subseteq L^2_{\sigma}(\R^3) \rightarrow L^2_{\sigma}(\R^3)$ from Section \ref{LINEAROPERATOR}. 

\subsection{The limiting case $c=0$}
We first consider the limiting case $c \rightarrow 0^+$, i.e., the linear operator that occurs by perturbing the equation
\begin{align*}
\partial_t u - \Delta u = |u|^{p-1}u, \quad t>0,
\end{align*} 
around the ODE solution $v_T(t,x) = (p-1)^{-\frac{1}{p-1}} (T-t)^{-\frac{1}{p-1}}$. The corresponding linear operator 
\begin{align*}
\tilde{\mathcal{L}} : \mathcal{D}(\mathcal{L}) \subseteq L^2_{\sigma}(\R^3) \rightarrow L^2_{\sigma}(\R^3), \quad [\tilde{\mathcal{L}} f](y) = \Delta f (y) - \frac{1}{2}y \cdot \nabla f(y) + f(y),
\end{align*}
for $f \in C^{\infty}_{c}(\R^3)$ can be decomposed in the same way as $\mathcal{L}$ with corresponding unitary equivalent one-dimensional Schr\"odinger operators $\tilde{\mathcal{B}}_0$ and $\tilde{\mathcal{B}}_{\ell m}$ with $\ell \geq 1$ and $m \in \{1,...,d(\ell)\}$, see the beginning of Section \ref{LINEAROPERATOR}. Since $\tilde{\mathcal{L}}$ also possesses the unstable eigenvalue $\lambda = 1$ that comes from the freedom of choosing the blowup time $T>0$ we can derive the supersymmetric partner $\tilde{\mathcal{B}}_S$ of $\tilde{\mathcal{B}}_0$ that is isospectral to $\tilde{\mathcal{B}}_0$ except the eigenvalue $\lambda = -1$. The underlying operators have the same domains as $\mathcal{D}(\mathcal{B}_S)$ and $\mathcal{D}(\mathcal{B}_{\ell m})$, respectively, and they are given by
\begin{align*}
[\tilde{\mathcal{B}}_S u](r) = -u^{\prime \prime}(r) + \left( \frac{r^2}{16} + \frac{2}{r^2} -\frac{5}{4} \right) u(r),
\end{align*}
and
\begin{align*}
[\tilde{\mathcal{B}}_{\ell m} u](r) = -u^{\prime \prime}(r) + \left( \frac{r^2}{16} + \frac{\ell ( \ell +1)}{r^2} -\frac{7}{4} \right) u(r),
\end{align*}
for $u \in C^{\infty}_c(\R^+)$. The eigenvalues of these operators are well-known and can be computed explicitly. In particular, for the unstable part of the spectrum one obtains
\begin{align*}
\{ \lambda \in \sigma(\tilde{\mathcal{B}}_0) : \lambda \leq 0 \} = \{0,-1 \} \quad \text{and} \quad \{ \lambda \in \sigma(\tilde{\mathcal{B}}_{1m}) : \lambda \leq 0 \} = \{ -\tfrac{1}{2} \},
\end{align*}
for $m \in \{ 1,2,3 \}$ and
\begin{align*}
\{ \lambda \in \sigma(\tilde{\mathcal{B}}_{2 m}) : \lambda \leq 0 \} = \{0 \} \quad \text{and} \quad \{ \lambda \in \sigma(\tilde{\mathcal{B}}_{ \ell m}) : \lambda \leq 0 \} = \emptyset,
\end{align*}
for $m \in \{1,...,d(\ell) \}$.
\subsection{The case $c >0$ - Numerical observations} We compare the eigenvalues of the operators $\mathcal{B}_0$ and $\mathcal{B}_{\ell m}$ that depend on the parameter $c$ to those of $\tilde{\mathcal{B}}_0$ and $\tilde{\mathcal{B}}_{ \ell m}$, respectively. To simplify the discussion we restrict ourselves to the case $p=3$ with $c \in ( 0, \frac{1}{3} )$. Numerical computations suggest that there is a bijection of all eigenvalues in the limiting case $c=0$ to those for $c \in (0, \frac{1}{3})$. More precisely, for $\mathcal{B}_0$ we see an eigenvalue $\lambda_c >0$ that tends to $0$ from above as $c \rightarrow 0^+$. The same holds true for $\mathcal{B}_{2m}$. This suggests that $\mathcal{B}_0$ has only one negative eigenvalue $\lambda = -1$ and that there are no negative eigenvalues for $\mathcal{B}_{2 m}$ where $m \in \{1,...,5 \}$. For $\ell \geq 3$, due to the positivity of the corresponding potentials, there exist no negative eigenvalues. It turns out that the remaining case $\ell =1$ is the most complicated one as the number of unstable eigenvalues appears to depend on $c$. \\
Since $\lambda = -\frac{1}{2}$ is an eigenvalue for $c =0$ we assume that at least for small $c >0$ there is a negative eigenvalue $\lambda_c$ close to $ -\frac{1}{2}$. However, as $c$ increases the corresponding potential $q_1$ in Eq.~\eqref{Bl} becomes positive. This indicates that there exists some critical value $c^* \in (0, \frac{1}{3} )$ such that for all $c \in (c^*, \frac{1}{3})$ the operator $\mathcal{B}_{1m}$ does not have unstable spectrum.  However, the complicated dependency of the potential on the parameter $c$ makes it difficult to rigorously determine the exact value of $c^*$. Eventually, for large enough $c$ the potential becomes positive, which is key in the proof of Theorem \ref{maintheorem2}.

\subsection{Application of GGMT}
To support our claim on the structure of the spectrum of $\mathcal{B}_{1 m}$ we apply the variant of the GGMT criterion provided in Appendix \ref{GGMTv}. For this, we rewrite the operator as 
\begin{align*}
[\mathcal{B}_{1 m} u](r) = -u^{\prime \prime}(r) + \frac{\delta-\frac{1}{4}}{r^2} +\left(\frac{r^2}{16} -\frac{1}{4} -V(r) + \frac{2 + \frac{1}{4} - \delta}{r^2} \right) u(r),
\end{align*}
for $u \in C^{\infty}_c(\R^+)$ and suitably small $\delta >0$. This operator satisfies the assumptions of Theorem \ref{TheoremA} for the potential
\begin{align*}
Q_{c,\delta}(r) := \frac{r^2}{16} -\frac{1}{4} -V(r) + \frac{2 + \frac{1}{4} - \delta}{r^2}.
\end{align*}
We define
\begin{align*}
G_{c,\delta}(\kappa):= \frac{(\kappa-1)^{\kappa-1} \Gamma(2 \kappa)}{(4 \delta)^{\kappa - \frac{1}{2}} \kappa^{\kappa} \Gamma(\kappa)^2} \int_0^{\infty} r^{2 \kappa -1} Q_{c,\delta-}(r)^{\kappa} dr.
\end{align*}
Numerical evaluation of the integral suggest that
\begin{align*}
G_{0.11,1}(1.5) \leq 0.92 \quad \text{and} \quad G_{0.1,1}(1.5) \geq 1.04,
\end{align*}
which implies that for $c =0.11$ there are no negative eigenvalues for $\ell = 1$ and for $c= 0.1$ there is at most one negative eigenvalue for $\ell =1$. The choice of $\delta =1$ for the criterion is based on the best result for $G_{c,\delta}$ we could accomplish. By calculating $G_{c,\delta}$ for further values of $c \in (0,\frac{1}{3})$ we conjecture that there is a $c^*$ close to $0.105$ such that for all choices of $\kappa \geq \frac 32$ and sufficiently small $\delta >0$ we get $G_{c,\delta}(\kappa) >1$ for all $c \in (0,c^*)$, while $G_{c,\delta}(\kappa) < 1$ for $c \in (c^*,\frac{1}{3})$ for some choice of $\delta$ and $\kappa$. 
\pagestyle{plain}
	\bibliography{Literatur}
	\bibliographystyle{plain}

\end{document}